\documentclass[11pt]{amsart}
\usepackage{amsmath}
\usepackage{amssymb}
\usepackage[english,french]{babel}
\usepackage{blindtext}
\usepackage{eucal}
\usepackage{amsthm}
\usepackage{amscd}
\usepackage{hyperref}
\usepackage{soul}
\usepackage{url}
\usepackage{skull}
\usepackage{adjustbox}

\usepackage{tikz}
\usetikzlibrary{shapes.geometric, arrows}
\usetikzlibrary{shapes,arrows}
\usepackage[latin1]{inputenc}
\usepackage{tikz}
\usetikzlibrary{shapes,arrows}
\usetikzlibrary{shapes.multipart}
\usepackage{verbatim}
\usepackage{tkz-berge}
\usepackage{tikz-qtree}

\usepackage{amssymb}
\usepackage[]{amsmath, amsthm, amsfonts,graphicx, amscd,}
\usepackage[all,cmtip]{xy}
\input amssym.def \input amssym
\usepackage{forest}
\usepackage{mdframed}
\usepackage{framed}

\begin{document}

\newtheorem{thm}{Theorem}[section]
\newtheorem{cor}[thm]{Corollary}
\newtheorem{claim}[thm]{Claim}
\newtheorem {fact}[thm]{Fact}
\newtheorem{con}[thm]{Conjecture}

\newtheorem*{thmstar}{Theorem}
\newtheorem{prop}[thm]{Proposition}
\newtheorem*{propstar}{Proposition}
\newtheorem {lem}[thm]{Lemma}
\newtheorem*{lemstar}{Lemma}
\newtheorem{conj}[thm]{Conjecture}
\newtheorem{question}[thm]{Question}
\newtheorem*{questar}{Question}
\newtheorem{ques}[thm]{Question}
\newtheorem*{conjstar}{Conjecture}
\newtheorem{fct}[thm]{Fact}
\theoremstyle{remark}
\newtheorem{rem}[thm]{Remark}
\newtheorem{exmp}[thm]{Example}
\newtheorem{cond}[thm]{Condition}
\newtheorem{np*}{Non-Proof}
\newtheorem*{remstar}{Remark}
\theoremstyle{definition}
\newtheorem{defn}[thm]{Definition}
\newtheorem*{defnstar}{Definition}
\newtheorem{exam}[thm]{Example}
\newtheorem*{examstar}{Example}
\newtheorem{assump}[thm]{Assumption}
\newtheorem{Thm}[thm]{Theorem}

\def \ta {\tau_{\mathcal{D}/\Delta}}
\def \D {\Delta}
\def \DD {\mathcal D}

\newcommand{\pd}[2]{\frac{\partial #1}{\partial #2}}
\newcommand{\td}{\text{tr.deg.}}
\newcommand{\pp}{\partial }
\newcommand{\pdtwo}[2]{\frac{\partial^2 #1}{\partial #2^2}}
\newcommand{\od}[2]{\frac{d #1}{d #2}}
\def\Ind{\setbox0=\hbox{$x$}\kern\wd0\hbox to 0pt{\hss$\mid$\hss} \lower.9\ht0\hbox to 0pt{\hss$\smile$\hss}\kern\wd0}
\def\Notind{\setbox0=\hbox{$x$}\kern\wd0\hbox to 0pt{\mathchardef \nn=12854\hss$\nn$\kern1.4\wd0\hss}\hbox to 0pt{\hss$\mid$\hss}\lower.9\ht0 \hbox to 0pt{\hss$\smile$\hss}\kern\wd0}
\def\ind{\mathop{\mathpalette\Ind{}}}
\def\nind{\mathop{\mathpalette\Notind{}}}
\numberwithin{equation}{section}

\def\id{\operatorname{id}}
\def\Frac{\operatorname{Frac}}
\def\Const{\operatorname{Const}}
\def\spec{\operatorname{Spec}}
\def\span{\operatorname{span}}
\def\exc{\operatorname{Exc}}
\def\Div{\operatorname{Div}}
\def\cl{\operatorname{cl}}
\def\mer{\operatorname{mer}}
\def\trdeg{\operatorname{trdeg}}
\def\ord{\operatorname{ord}}

\newcommand{\m}{\mathbb }
\newcommand{\mc}{\mathcal }
\newcommand{\mf}{\mathfrak }
\newcommand{\is}{^{p^ {-\infty}}}
\newcommand{\QQ}{\mathbb Q}
\newcommand{\fh}{\mathfrak h}
\newcommand{\CC}{\mathbb C}
\newcommand{\RR}{\mathbb R}
\newcommand{\ZZ}{\mathbb Z}
\newcommand{\tp}{\operatorname{tp}}
\newcommand{\SL}{\operatorname{SL}}
\newcommand{\gen}[1]{\left\langle#1\right\rangle}


\title[Some functional transcendence results]{Some functional transcendence results around the Schwarzian differential equation.}

\author[D. Bl\'azquez-Sanz]{David Bl\'azquez-Sanz}
\address{David Bl\'azquez-Sanz, Universidad Nacional de Colombia - Sede Medell\'in, Facultad de Ciencias, Escuela de Matem\'aticas, Colombia}
\email{ dblazquezs@unal.edu.co}

\author[G. Casale]{Guy Casale}
\address{Guy Casale, Univ Rennes, CNRS, IRMAR-UMR 6625, F-35000 Rennes, France}
\email{guy.casale@univ-rennes1.fr}

\author[J. Freitag]{James Freitag}
\address{James Freitag, University of Illinois Chicago, Department of Mathematics, Statistics,
and Computer Science, 851 S. Morgan Street, Chicago, IL, USA, 60607-7045.}
\email{jfreitag@uic.edu}

\author[J. Nagloo]{Joel Nagloo}
\address{Joel Nagloo, CUNY Bronx Community College, Department of Mathematics and Computer Science, Bronx, NY 10453, and
CUNY Graduate Center, Ph.D. programs in Mathematics, 365 Fifth Avenue,
New York, NY 10016, USA}
\email{joel.nagloo@bcc.cuny.edu}

\thanks{G. Casale is partially supported by COFECUB-41988WC ``Holomorphic Foliations''. J. Freitag is partially supported by NSF grant DMS-1700095 and NSF CAREER award 1945251. J. Nagloo is partially supported by NSF grants DMS-1700336 and DMS-1952694, and PSC-CUNY grant \#63304-00 51.}

\maketitle
\selectlanguage{english}
\begin{abstract} This paper centers around proving variants of the Ax-Lindemann-Weierstrass (ALW) theorem for analytic functions which satisfy Schwarzian differential equations. In previous work, the authors proved the ALW theorem for the uniformizers of genus zero Fuchsian groups, and in this work, we generalize that result in several ways using a variety of techniques from model theory, differential Galois theory and complex geometry. 
\end{abstract}

\selectlanguage{french}
\begin{abstract} 
Dans cet article, nous d\'emontrons des variantes du th\'eor\`eme d'Ax-Lindemann-Weierstrass (ALW) pour des fonctions analytiques satisfaisant des \'equations diff\'erentielles de type `Schwarzienne'. Dans des travaux ant\'erieurs, nous avons prouv\'e le th\'eor\`eme ALW pour les uniformisantes de groupes fuchsiens de genre z\'ero. Dans ce travail, nous g\'en\'eralisons ce r\'esultat de plusieurs mani\`eres en utilisant des techniques vari\'ees provenant de la th\'eorie des mod\`eles, de la th\'eorie de Galois diff\'erentielle et de la g\'eom\'etrie complexe.
\end{abstract}
\selectlanguage{english}
\section{Introduction}
Let $X$ and $Y$ be algebraic varieties over $\m C$ and let $\phi: X^{an} \rightarrow Y^{an}$ be a complex analytic map which is not algebraic. In this case, for most algebraic subvarieties $X_0 \subset X$, the image $\phi (X_0)$ is \emph{not} algebraic. The pairs of algebraic subvarieties $(X_0 , Y_0)$ with $X_0 \subset X$ and $Y_0 \subset Y$ such that $\phi (X_0) = Y_0$ are called \emph{bi-algebraic} for $\phi$. Bi-algebraic subvarieties should be rare and revealing of important geometric aspects of the analytic map $\phi$. This manuscript centers around the problem of determining the bi-algebraic subvarieties of analytic maps and several related problems of functional transcendence. The maps we consider satisfy nonlinear differential equations of a certain general form.

The condition that $X$ is an algebraic variety is in fact slightly too restrictive for many of the specific interesting examples both here and in the literature, and so generally we will allow $X$ to be an open subset of an algebraic variety. 
Then an algebraic subvariety of $X$ is a set given by the vanishing of a finite system of polynomial equations on the open set. We will be especially interested in the case that $X$ is the universal cover of $Y$, where open domains such as $\m H, $ the complex upper half-plane arise naturally.  For a survey of recent developments in this area mainly centering on approaches using o-minimality, see \cite{KUYsurvey}. Our approach, started in \cite{CasFreNag}, is much different. 


We approach the problem through studying the differential equations satisfied by the function $\phi.$ Then the bi-algebraic subvarieties correspond in a natural way to algebraic relations between solutions of systems of differential equations. The classification of such relations for a given system is one of the central preoccupations of differential Galois theory and the model theory of differential fields, two of the central tools we employ. In \cite{CasFreNag}, we solved the bi-algebraicity problem with $\phi$ given by the map applying $j_\Gamma$ and its first two derivatives to any number of coordinates in $\m H^n$, where $j_\Gamma$ is a uniformizing function associated with the quotient $\Gamma \backslash \m H$ where $\Gamma$ is a Fuchsian group of the first kind and genus zero, and therefore $\Gamma \backslash \m H$ is a $n$-punctured sphere. In this case, denoting a coordinate in the domain by $t$, we have that $j_\Gamma (t)$ is a solution of the \emph{Schwarzian equation}:
$$S_{t}(y) + \frac{1}{2} (y')^2 R_\Gamma(y) = 0, $$
where $R_\Gamma$ is a rational function with coefficients in $\m C$, $y' = \frac{dy}{dt}$, and $S_{t}(y) = \frac{y'''}{y'} -\frac{3}{2}\left(\frac{y''}{y'}\right)^2$ denotes the Schwarzian. In this paper we wish to consider the bi-algebraicity problem for solutions of an arbitrary (no assumption on the rational function $R$) Schwarzian equation:  
$$S_{t}(y) + \frac{1}{2} (y')^2 R(y) = 0. $$
In the case of a Fuchsian group $\Gamma$, the rational function $R_\Gamma$ depends on the group $\Gamma$ and the choice of a uniformizing function $j_\Gamma$, and as there are only countably many Fuchsian groups of the first kind of genus zero, the results of \cite{CasFreNag} only apply to a very particular (and countable) class of rational functions $R$. The bi-algebraicity problem has an equivalent statement in terms of functional transcendence, see \cite{PilaAO,CasFreNag}. We will, following \cite{PilaAO}, refer to both forms as  Ax-Lindemann-Weierstrass type theorems. 

In this paper, we make two significant steps towards the solution of the general bi-algebraicity problem for analytic functions satisfying Schwarzian equations. In Section \ref{GenericTri}, we consider those $R$ which are of the same general form as those in the Schwarzian equations satisfied by analytic functions which are uniformizers of Fuchsian triangle groups. In that case, the function $R_\Gamma (y)$ takes the form

\[R_{\alpha,\beta,\gamma}(y)=\frac{1}{2}\left(\frac{1-{\beta}^{-2}}{y^2}+\frac{1-{\gamma}^{-2}}{(y-1)^2}+\frac{{\beta}^{-2}+{\gamma}^{-2}-{\alpha}^{-2}-1}{y(y-1)}\right)\] where $\alpha, \beta , \gamma$ correspond to the angles of the triangle which is a fundamental half domain for the triangle group $\Gamma$. In Section \ref{GenericTri} we consider first the case of ``generic Schwarz triangle equations" where $\alpha, \beta, \gamma \in \m C$ are algebraically independent over $\m Q.$ In this case, we give a complete solution to the problem of bi-algebraicity, even with different such analytic functions applied to each coordinate. The case of generalized triangle equations is generally interesting (it includes, for instance the ALW result for the $j$-function associated with elliptic curves), but it also allows for an important and interesting generalization of our ALW result from \cite{CasFreNag}, which we describe next. 

For a Fuchsian group of the first kind $\Gamma$ the quotient $\Gamma \backslash \m H$ can be represented as an algebraic curve, not necessarily complete. By the genus of the group $\Gamma$ we mean the genus of the unique non singular completion of $\Gamma \backslash \m H$. In \cite{CasFreNag} we proved the ALW theorem for $\Gamma \backslash \m H$ genus zero.  In Section \ref{belyi} we drop the assumption that $\Gamma \backslash \m H$ is genus zero, but assume that $\Gamma \backslash \m H$ is given by an algebraic curve over $\m Q^{alg}$. Belyi \cite{Belyi} proved that any nonsingular projective algebraic curve over $\m Q^{alg}$ gives a cover of the Riemann sphere which is ramified at only three points. Belyi's theorem allows us to apply our result for triangle groups to prove the ALW theorem in the case that $\Gamma \backslash \m H$ has arbitary genus but is an algebraic curve over $\m Q^{alg}$. We leave the general case in which $\Gamma \backslash \m H$ is not assumed to be defined over $\m Q ^{alg}$ as an open problem for future work. In Proposition \ref{FiniteIndex}, we also establish a nice general fact showing that ALW results are not sensitive to finite index changes in $\Gamma$, a result used implicitely in \cite{FreSca} for the modular group.

In our previous work, we intensively studied the Schwarzian equation with $R_\Gamma$ the rational function coming from a genus zero Fuchsian group of the first kind. In Section \ref{GenericTri} we generalized our work under various assumptions on the form of the rational function $R$ or of the field of definition of $\Gamma \backslash \m H$, but pursued similar conclusions as in \cite{CasFreNag}. The second setting we consider in Section \ref{bialgcurves} assumes only a generic hypothesis about the rational function $R$ (that the Riccati equation associated with the rational function $R$ has no algebraic solutions, this is related to the differential Galois group of the linearized Schwarzian equation and it is satisfied by all uniformization equations), a much more general setting than any of the previously mentioned work. In this very general setting, we obtain slightly weaker results by characterizing bi-algebraic \emph{curves} rather than all varieties. Our results also allow for understanding the bi-algebraic curves for analytic maps $(J_1,J_2)$ where $J_1, J_2$ are solutions to Schwarzian equations associated with different rational functions $R_1, R_2.$ The developments of Section \ref{bialgcurves} are interesting functional transcendence results in their own right and point to the possibility of developing generalizations of our results in \cite{CasFreNag} and Section \ref{GenericTri}. 

\subsection*{Acknowledgements}
This work was mainly done at the School of Mathematics at the Institute for Advanced Study as a part of the 2019 Summer Collaborators Program. We thank the Institute for its generous support and for providing an excellent working environment. The authors also thank Peter Sarnak for useful conversations during our time at IAS. We also would like to thank the anonymous referee for careful reading of the paper and for very helpful comments and suggestions.

\section{Model theory and differential algebra} 

In this section, we provide a concise introduction to the model theory of differential fields, which we use in an essential way throughout much of the paper. Our introduction is not meant to be comprehensive, but we hope it provides the context and references for readers without a background in this area. A \emph{differential ring} $R$ is a commutative unitary ring with a \emph{derivation} ${\partial} : R \rightarrow R.$ A derivation ${\partial}:R \rightarrow R$ is an additive homomorphism satisfying Leibniz rule: $${\partial} (xy) = x {\partial} (y) + {\partial} (x ) y.$$ Any differential ring has a subring of \emph{constants} $C_R := \{x \in R \, | \, {\partial}x =0 \}$. 

When $R$ is a differential ring, by $R\{ \bar x \}$, we denote the ring of differential polynomials in variables $\bar x = (x_1, \ldots , x_n ).$ $R\{ \bar x \}$ is (as a ring) the ring of polynomials in $\bar x$ and their derivatives of every order over $R$, equipped with the natural derivation extending the derivation on $R$. When $K$ is a differential field and $F/K$ is an extension of differential fields with $a \in F$, by $K \langle a \rangle $ denotes the differential field generated by $a$ over $K$. All of the fields considered in this text will be of characteristic zero. 

Differential algebraic geometry - studying solution sets of systems of differential polynomials over differential rings - along the lines of classical algebraic geometry was largely initiated by Ritt (see \cite{Ritt1,KolchinDAAG}). The connection to model theory began with Robinson \cite{abby}, who axiomatized \emph{differentially closed fields of characteristic zero}, $DCF_0.$ We give the simpler axiomatization of Blum \cite{Blum}: 
\begin{itemize}
\item The axioms of algebraically closed fields of characteristic zero. 
\item $\partial$ is a derivation. 
\item Given two nonconstant differential polynomials $f$, $g$ in a single variable $x$, such that the order of $g$ is less than the order of $f$, there is some $x$ such that $f(x) = 0 $ and $g(x) \neq 0.$ 
\end{itemize} 

Differential fields $K \models DCF_0$ satisfying the axioms were shown by Blum to have \emph{quantifier elimination} (Robinson shows a slightly weaker property called \emph{model completeness}). For the reader not familiar with the terminology of model theory, we will give a geometric explanation of quantifier elimination in this setting. Any differential field $K$ comes equipped with the \emph{Kolchin topology}, in which the closed subsets of $K^n$ are the zero sets of systems of differential polynomials in $n$ variables. The Ritt-Raudenbush Basis Theorem implies that this topology is Noetherian and every closed set is given as the zero set of a finite system of differential polynomial equations. Such closed sets are called \emph{affine} differential varieties. A \emph{constructible set} in the Kolchin topology is finite Boolean combination of closed sets. In this setting of $DCF_0$, \emph{quantifier elimination} simply means that coordinate projections of constructible sets are again constructible. Concretely, given a system of differential polynomial equations\footnote{Any system of differential polynomial equations and inequalities is equivalent to one of the form given.} 
\begin{eqnarray*}
    p_1 (\bar x , \bar y ) & =   0 \\
    & \vdots  \\
    p_k (\bar x , \bar y ) & =  0 \\
    q(\bar x , \bar y) & \neq   0,
\end{eqnarray*}
then the set $\{\bar x \in K \, | \, \exists \bar y , \, p_1(\bar x, \bar y ) = \ldots = p_k ( \bar x , \bar y ) = 0 , \, q ( \bar x , \bar y ) \neq 0 \}$ is also given by some system of differential polynomial equations and inequalities in $\bar x$ over $K$. 

Above, we described the axioms of differentially closed fields and quantifier elimination informally using systems of algebraic differential equations. Next, we will describe the model theoretic perspective more formally. Our discussion assumes some familiarity with the basic notion of a first order formula in language, and we direct the reader to Section 1 of the very good survey article \cite{Marker} for more details about the basics of model theory.

The model theory of differential fields employs the language of differential rings, 
$$L_{\partial}=(0,1,+,\cdot,\partial).$$ 
The above axioms of Blum can be written as \emph{first order sentences} in the language $L_\partial$, while the axioms of algebraically closed fields (used as part of Blum's axioms) can be written as first order sentences in the simpler language of rings, $$L=(0,1,+,\cdot).$$
 
If $K$ is a differentially closed field, $F$ a subfield, $\overline{b}\in F^{m}$ and $\phi(\bar x ,\bar y)$ is an $L_{\partial}$-formula with free variables $\overline{x}$ and $\overline{y}$, then the set $X$ defined by 
\begin{eqnarray*}
X&:=&\{\overline{a}\in K^{n}\;:\;\phi(\overline{a},\overline{b})\text{ is true }\}\\
&=&\{\overline{a}\in K^{n}\;:\;K\models\phi(\overline{a},\overline{b})\}
\end{eqnarray*}

is called a {\em definable set} with parameters (i.e., the tuple $\overline{b}$) from $F$ or an $F$-definable set for short. For the reader uncomfortable with the general notion of a definable set, recall, in this setting, by quantifier elimination, $\phi (\bar x , \bar y)$ can be assumed to be given by a system of differential polynomial equations and inequalities (as in the discussion above). 

We now assume that $\mc U$ is a \emph{universal differential field} in the sense of \cite{Kolchin1953}; equivalently, we assume that $\mc U$ is a  saturated differentially closed field in the sense of model theory \cite[Section 6]{Marker}. So all differential fields discussed below are assumed to embed in $\mc U$. Let $\mc C \subset \mc U$ denote the constant field - in later sections of the paper, we will often assume that $\mc C$ is the field of complex numbers (often, for the results we pursue, this assumption entails no loss of generality). 

\begin{fct}\label{stableembed}
The field $\mc C$ is algebraically closed and any definable subset of $\mc C^n$ definable with $L_{\partial}$-formulas using parameters from a differential field $F$ is definable with $L$-formulas using parameters from $\mc C_F$. The $F$-definable subsets of $\mc C^n$ are precisely the $\mc C_F$-constructible subsets of $\mc C^n$ in the Zariski topology. 
\end{fct}
A proof can be found in \cite[Lemma 1.11]{LAMT} and \cite[Fact 1.6]{KamPil}. From Fact \ref{stableembed} together with quantifier elimination, one can obtain the following result which we use extensively in later sections:

\begin{fct}\label{descent}
Let $\theta(x_1,x_2,\ldots,x_n)$ be a $L_{\partial}$-formula with parameters in a subfield field $F$ of $\mc C$. Suppose we have $(\alpha_1,\alpha_2,\ldots,\alpha_n ) \in \mc C^n$ which is a generic point on some irreducible algebraic variety $V$ over $F$ such that $\mathcal{U} \models \theta (\alpha _1,\alpha _2,\ldots,\alpha _n)$. Then for an $F$-definable (thus Zariski-constructible) Zariski dense subset $U$ of $V$, we have $\mc U \models \theta (\bar a )$ for all $\bar a \in U$.
\end{fct} 
To see that this fact holds,\footnote{This fact is mild a generalization of \cite[Fact 2.11]{NagPil1}.} note that the collection W of elements $\bar a \in \mc C^n$ such that $\mc U \models \theta ( \bar a )$ is an $F$-definable subset of $\mc C^n$. By Fact \ref{stableembed}, we have that $W$ is an $F$-constructible subset of $\mc C^n$ in the Zariski topology. But since the formula $\theta$ holds for a generic point $\bar \alpha$ of $V$ over $F$, the set $U=W\cap V$ must be Zariski dense in $V$. 

Both of the previous facts center around establishing structural results around definable sets in $DCF_0$. Developing a deep structure theory for definable sets (including e.g. dimension theory and understanding invariants of definable sets) is a central goal in most parts of model theory. Such a structural theory is, in general, impossible to develop due to certain phenomena which arise in general in first order logic (e.g. incompleteness and undecidability). The approach of much of modern model theory is to make some \emph{tameness} assumption about the class of structures one considers, and then to attempt to understand definable sets in the more restrictive setting. For instance, a structure is \emph{o-minimal} if it is totally ordered and the definable sets in one variable are finite unions of points and open intervals. Many disparate modern applications of model theory share this same very general theme.

With the new axiomatization of $DCF_0$, Blum \cite{Blum} was able to observe that $DCF_0$ satisfies a strong structural property, called \emph{$\omega$-stability}. So, deep results from that setting can be brought to bear in understanding constructible sets in the Kolchin topology. The previous two facts can be seen as a special case of a more general principle called \emph{stable-embeddedness} which holds in even more general settings \cite{Pillaystabem}. It is not feasible to give a comprehensive and accessible overview of $\omega$-stability, but in the portions the next section of the paper where the theory is utilized, we attempt to give a self-contained statement and explanation of the relevant results as the arise - for instance, in Section \ref{ortho}. The notions we use from geometric stability theory can be found in the comprehensive reference \cite{GST}. 

\section{Summary of the Genus zero case}
\label{section2}

Let us begin with a quick review of the results in \cite{CasFreNag}. We recall them here so that the paper is self-contained.

\subsection{Schwarzian Equation} By a Schwarzian differential equation, we mean an equation of the form
\begin{equation}\tag{$\star$}
S_{t}(y) +(y')^2 R(y) =0\label{(*)}.
\end{equation}
where $S_{t}(y)=\left(\frac{y''}{y'}\right)' -\frac{1}{2}\left(\frac{y''}{y'}\right)^2$ denotes the Schwarzian derivative ($'=\frac{d}{dt}$) and $R$ is a rational function over $\mathbb{C}$. The equation naturally appears in the study of automorphic functions. Indeed, the Riemann mapping theorem states that if $D$ is a non-empty simply connected open subset of $\mathbb{C}$ which is not all of $\mathbb{C}$, then there exists a biholomorphic mapping $f$ from $D$ onto the upper half complex plane $\mathbb{H}$. Furthermore, if $D$ is bounded by a simple closed contour, then $f$ extends to a homeomorphism from the closure of $D$ onto $\overline{\mathbb{H}}=\mathbb{H}\cup{\bf P}^1(\mathbb{R})$. 
\begin{center}
\includegraphics[scale=0.3]{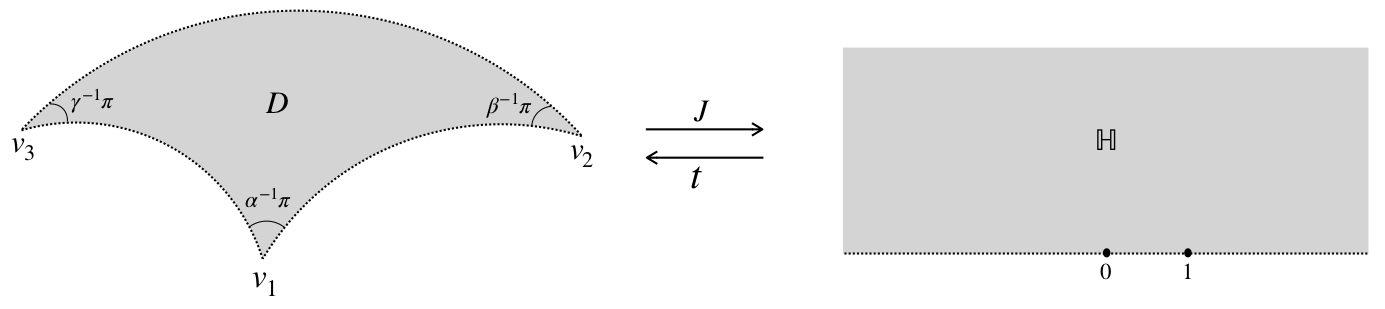}    
\end{center}
In the case when $D$ is a circular polygon with vertices $v_{1},v_{2},\ldots,v_{n}$ and with respective internal angles $\frac{\pi}{\alpha_1},\frac{\pi}{\alpha_2},\ldots,\frac{\pi}{\alpha_2}$, Schwarz (cf. \cite[Section 5.8]{AbloFoka}) showed that there exist $2n$ real numbers $a_1,\ldots,a_n$ and $\beta_1,\ldots,\beta_n$ such that the (unique up to action of ${\rm PSL}_2(\mathbb R)$) biholomorphic mapping $J:D\rightarrow\overline{\mathbb{H}}$ satisfies a Schwarzian equation $(\star)$ with $R$ given by
\[R_{J}(y)=\frac{1}{2}\sum_{i=1}^{r}{\frac{1-\alpha_i^{-2}}{(y-a_i)^2}}+\sum_{i=1}^{r}{\frac{\beta_i}{y-a_i}}.\]

\begin{exmp}\label{ExampleTriangle}
In the special case when $D$ is a circular triangle $\bigtriangleup(\alpha,\beta,\gamma)$ with vertices $v_{1},v_{2},v_{3}$ and with respective internal angles $\frac{\pi}{\alpha}$, $\frac{\pi}{\beta}$ and $\frac{\pi}{\gamma}$, one can completely determine the constants appearing in the equations. Indeed, if we impose that $J$ sends the vertices $v_{1},v_{2},v_{3}$ to $\infty,0,1$ respectively, then 
\[R(y)=R_{\alpha,\beta,\gamma}(y)=\frac{1}{2}\left(\frac{1-{\beta}^{-2}}{y^2}+\frac{1-{\gamma}^{-2}}{(y-1)^2}+\frac{{\beta}^{-2}+{\gamma}^{-2}-{\alpha}^{-2}-1}{y(y-1)}\right).\]
The function $J(t)$ (as well as its inverse) is called a {\em Schwarz triangle function}. On the other hand, by a {\em Schwarz triangle equation} we mean the equation ($\star$) where the parameters $\alpha,\beta,\gamma$ are any complex numbers. By a generic Schwarzian triangle equations we mean the ODE ($\star$) with $R = R_{\alpha,\beta,\gamma}$ and such that $\alpha,\beta,\gamma$ are algebraically independent over $\mathbb{Q}$.
\end{exmp}


Some differential algebraic properties of solutions of equation ($\star$) can be understood by means of its linearization. Let 
us consider the following second order linear differential equation
\begin{equation}\label{eq:linear}
\frac{d^2\psi}{dy^2} + \frac{1}{2}R(y)\psi=0.
\end{equation}
A direct substitution shows that the logaritmic derivative $u = \frac{d\log\psi}{dy}$ of any solution satisfies the Riccati equation
\begin{equation}\label{eq:riccati}
\frac{du}{dy} + u^2 + \frac{1}{2}R(y) = 0.
\end{equation}
Moreover, if $\psi_1$ and $\psi_2$ are linearly independent solutions of \eqref{eq:linear} then the quotient $t = \frac{\psi_1}{\psi_2}$ satisfies $\frac{d^2 t}{dy^2} = -2u\frac{dt}{dy}$
and it follows,
\begin{equation}\label{eq:principal}
S_{y}(t) = R(y).
\end{equation}
From the inversion formula for Schwarzian derivatives we obtain that the inverse function of $t$ is a solution of ($\star$). The definition of $t$ as the quotient of two linearly independent solutions of \eqref{eq:linear} ensures that the group of ${\rm PSL}_2(\CC)$ acts freely and transitively in the space of solutions of \eqref{eq:principal}: if $t_1$ and $t_2$ are two such solutions then
$t_2 = \frac{at_1 + b}{ct_1 + d}$ for certain $a,b,c,d\in \mathbb C$ with $ad-bc = 1$.
Let us fix a local solution $y(t)$ for $(\star)$ and let $t(y)$ be its inverse function. Let us denote by $\dot t$, $\ddot t$ the derivatives of $t$ with respect to $y$. We have
$$y' = \frac{1}{\dot t(y)}, \quad y''= -\frac{\ddot t(y)}{\dot t^3(y)}.$$
 Let us consider $\mathbb C(t)\langle y \rangle_{\frac{d}{dt}}$ the field of the variety defined by equation \eqref{(*)}, it is $\mathbb C(t,y,y',y'')$ endowed with the derivation 
$$\frac{d}{dt} =\frac{\partial}{\partial t} + y'\frac{\partial}{\partial y} +y''\frac{\partial}{\partial y'}+ \left( \frac{3}{2} \frac{y''^2}{y'} - y'^3 R(y) \right)\frac{\partial}{\partial y''}.$$
 If one does the same construction for the equation \eqref{eq:principal} we get a differential field $\mathbb C(y)\langle t \rangle_{\frac{d}{dy}}$ described by the field $\mathbb C(y,t,\dot t, \ddot t)$ with derivation
 $$
 \frac{d}{dy} = \frac{\partial}{\partial y} + \dot t \frac{\partial}{\partial t} + \ddot t \frac{\partial}{\partial \dot t} + \left(\frac{3}{2} \frac{\ddot t^2}{\dot t} + R(y)\dot t \right)\frac{\partial}{\partial \ddot t}.
 $$ 
 By the above considerations we have an identity of fields (but not differential fields),
$$
\mathbb C(t)\langle y\rangle_{\frac{d}{dt}} = \mathbb C(t,y,y',y'') = \mathbb C(y,t,\dot t,\ddot t) = \mathbb C(y)\langle t\rangle_{\frac{d}{dy}}.$$
Such aforementioned differential structures are, of course, related. Both represent the same differential equation with different choices of independent variables; they are proportional, $\frac{d}{dt} = y'\frac{d}{dy}$. The next proposition summarizes the elements of differential Galois theory needed in this article  for more detail the reader may consult the booklet \cite{magid}. 
\begin{prop}\label{riccati}
The differential field extension $\mathbb C(y)\subset \mathbb C(y)\langle t \rangle$ is a Picard-Vessiot extension with Galois group in ${\rm PSL}_2(\mathbb C)$. The following are equivalent
\begin{enumerate}
    \item $t,y,y',y''$ are algebraically independent over $\mathbb C$.
    \item The differential Galois group  ${\rm Aut}(\mathbb C(y)\langle t \rangle/\mathbb C(y))$ is ${\rm PSL}_2(\mathbb C)$.
    \item The Riccati equation \eqref{eq:riccati} has no solution in $\mathbb C(y)^{\rm alg}$.
\end{enumerate}
\end{prop}

\begin{proof}
Let us write $t$ as the quotient $t = \frac{\psi_1}{\psi_2}$ of two solutions of \eqref{eq:linear}.
Then $\mathbb C(y)\subset\mathbb C(y)\langle \psi_1, \psi_2 \rangle$ is a Picard-Vessiot extension.
Its differential Galois group, 
${\rm Aut}(\mathbb C(y)\langle \psi_1, \psi_2 \rangle/\mathbb C(y))$, is represented in ${\rm SL}_2(\mathbb C)$ as a group of special linear matrices,
$$\left[\begin{array}{c} \psi_1 \\ \psi_2 \end{array} \right]\mapsto \left[\begin{array}{cc} a & b \\ c & d \end{array} \right]\left[\begin{array}{c} \psi_1 \\ \psi_2 \end{array} \right].$$
It follows, from Galois correspondence, that the intermediate extension $\mathbb C(y)\subset \mathbb C(y)\langle t \rangle$ 
is a Picard-Vessiot extension whose differential Galois group ${\rm Aut}(\mathbb C(y)\langle t \rangle/\mathbb C(y))$ is the image of ${\rm Aut}(\mathbb C(y)\langle \psi_1, \psi_2 \rangle/\mathbb C(y))$ in ${\rm PSL}_2(\mathbb C)$. 

 Let us see the equivalence. $(1)\Rightarrow(2)$: 
in such case the transcendence degree of $\mathbb C(y)\subset \mathbb C(y)\langle t\rangle$ 
is $3$, so its differential Galois group is the only $3$-dimensional group of ${\rm PSL}_2(\mathbb C)$. 
$(2)\Rightarrow(3)$: as it is shown in the preliminary considerations to Kovacic's algorithm; see \cite[\S 1.3]{Kovacic} the differential Galois group of 
$\mathbb C(y)\subset\mathbb C(y)\langle \psi_1, \psi_2 \rangle$ is a proper subgroup of ${\rm SL}_2(\mathbb C)$ if and only if the associated Riccati equation
has an algebraic solution. Having into account that the image of any proper 
subgroup ${\rm SL}_2(\mathbb C)$ is ${\rm PSL}_2(\mathbb C)$ is also a proper 
subgroup we finish. $(3)\Rightarrow (1)$: by the same reasons, the differential Galois group of $\mathbb C(y)\subset\mathbb C(y)\langle \psi_1, \psi_2 \rangle$ is ${\rm SL}_2(\mathbb 
C)$. Thus, by Galois correspondence the differential Galois group of  ${\rm Aut}(\mathbb 
C(y)\langle t \rangle/\mathbb C(y))$ is ${\rm PSL}_2(\mathbb C)$ of dimension $3$,
its transcendence degree over $\mathbb C(y)$ is thus $3$ and then $t,y,y',y''$ are
algebraically independent over $\mathbb C$.
\end{proof}

\subsection{Irreducibly \`a-la-Umemura and strong minimality} \label{someGST}

We fix a universal differentially closed field $\mathcal U$ and assume that the field of constants of $\mc U$ is $\m C$.

\begin{defn}
A definable set $\mathcal{Y}$ is said to be \emph{strongly minimal} if it is infinite and every definable subset is finite or co-finite.
\end{defn}

Strong minimality is a central notion in model theory for both applications and theoretical reasons. For instance, the Baldwin-Lachlan theorem gives a characterization of uncountably categorical theories in terms of  strongly minimal sets \cite{BaldLach}. If the definable set $\mathcal{Y}$ is defined by an ODE of the form $y^{(n)}=f(t,y,y',\ldots,y^{(n-1)})$, where $f$ is rational over $\mathbb{C}$ (like in the case of the Schwarzian equations), then $\mathcal{Y}$ is strongly minimal if and only if for any differential field extension $K$ of $\mathbb{C}$ and solution $y\in\mathcal{Y}$ , $\text{tr.deg.}_KK\gen{y}=0$ or $n$. Strong minimality has many interesting consequences on the relations between solutions of the equation. 

\begin{defn} Let $\mathcal{Y}$ be a strongly minimal set defined by an order $n$ ODE over $\m C$. Then
\begin{enumerate}
\item $\mathcal{Y}$ is \emph{geometrically trivial} if for any differential field extension $K$ of $\mathbb{C}$ and for any distinct solutions $y_{1},\ldots,y_{m}$, if the collection consisting of $y_{1},\ldots,y_{m}$ together with all their derivatives $y_{i}^{(j)}$ up to order $n-1$ is algebraically dependent over $K$ then for some $i<j$, the collection $y_{i}, y_{j}$ together with their derivatives is algebraically dependent over $K$. 
\item $\mathcal{Y}$ is \emph{strictly disintegrated} if for any differential field extension $K$ of $\mathbb{C}$ and $y_1,\ldots,y_n$ are distinct solutions that are not algebraic over $K$, then $$tr.deg._KK(y_1,y'_1,y''_1,\ldots,y_n,y'_n,y''_n)=3n$$.
\end{enumerate}
\end{defn}
Any strongly minimal set can be given the structure of a matroid (or pregeometry) and the notion of geometric triviality has its origin from matroid theory. An exposition of this very general model theoretic phenomenon and its application to $DCF_0$ can be found in \cite{MarkerSM}. The papers \cite{NagPil1} and \cite{NagPil2} discuss strict disintegratedness and the related notion {\em $\omega$-categoricity} in more details.

For order 1 definable sets defined over $\m C$, (so in particular autonomous differential equations), strong minimality implies geometric triviality.

\begin{fct}{\cite[Proposition 5.8]{CasFreNag}}\label{geo}
Let $\mathcal Y$ be a strongly minimal set of order $>1$ and suppose that $\mathcal Y$ is defined over $\mathbb C$. Then $\mathcal Y$ is geometrically trivial.
\end{fct}

Schwarzian equations are autonomous and we aim to prove that algebraic relation between theirs solutions have very specific arithmetic origins. Proposition \ref{riccati} can be seen as the $0$-step version of strong minimality theorem below has it involves no field extensions.

\begin{fct}{\cite[Theorem 3.2]{CasFreNag}} \label{thm:strongminimality}
Let $(K,\partial)$ be any differential field extension of $\mathbb{C}$ and let us assume that the Riccati equation (\ref{eq:riccati}) has no algebraic solution. For any solution $j_R$ of the Schwarzian equation $(\star)$ we have that
\[\text{tr.deg.}_KK\gen{j_R}=0\text{ or }3.\] In other words, the set defined by the equation $(\star)$ is strongly minimal.
\end{fct}

A particular case of strong minimality already appear in the work of H. Umemura on the irreducibility of Painlev\'e equations \cite{Umemura1988, Umemura1990}. Following the classification of transcendency of solutions of differential equations started by Painlev\'e in his 21st {\it Le\c{c}on de Stockholm} \cite{painleve}, Umemura defined the notion of irreducible equation (\cite{Umemura1988}) and proved the irreducibility of the first Painlev\'e equation. Then, for a second order differential equation without algebraic solutions, irreducibility is a consequence of the condition (J) (see \cite[page 169]{Umemura1990}). 

Umemura's condition (J) is equivalent to strong minimality. Umemura's theorem is stated in the case of second order differential equation. It is not hard to generalize Umemura's work to higher order equations. A general statement is given by: {\it If a differential equation defines a strongly minimal set, then its generic solution is not contained in a differential field obtained by successive iteration of strongly normal extensions and extensions by solutions of lower order differential equation.} While not published anywhere, all the ingredients needed to prove this statement can be found in \cite[Appendix A]{NagPil2} and \cite[Chapter 3]{Nag}.

\subsection{Fuchsian triangle groups}\label{trainglegp} We will now describe the main object of study of the paper \cite{CasFreNag}. We will restrict ourselves to Fuchsian triangle groups. For details about other Fuchsian groups of genus zero, we direct the reader to Section 2 of that paper. 

Let $\Gamma_{(k,l,m)}\subset {\rm PSL}_2(\mathbb{R})$ be a Fuchsian triangle group\footnote{In the remainder of the text all triangle groups will be assumed Fuchsian.}, that is assume that $\Gamma_{(k,l,m)}$ is a Fuchsian group of the first kind with signature is $(0;k,l,m)$. It is known that $\Gamma_{(k,l,m)}$ is a subgroup of index 2 of the group generated by the reflections in the sides of a {\em hyperbolic} triangle $\bigtriangleup(k,l,m)$; namely circular triangles such that the parameters $k,l,m$ satisfying the relation $\frac{1}{k}+\frac{1}{l}+\frac{1}{m}<1$ and such that $k,l,m\in\mathbb{N}\cup\{\infty\}$. The group $\Gamma_{(k,l,m)}$ has the following presentation
\[\Gamma_{(k,l,m)}=\gen{g_1,g_2,g_3\in {\rm PSL}_2(\mathbb{R})\;:\;g_1^{k}=g_2^{l}=g_3^{m}=g_1g_2g_3=I}\] 
and acts on $\mathbb{H}$ by linear fractional transformation: for 
$\begin{pmatrix}
    a & b \\
    c & d
  \end{pmatrix}
  \in \Gamma_{(k,l,m)}$ and $\tau\in\mathbb{H}$
\[ \begin{pmatrix}
    a & b \\
    c & d
  \end{pmatrix}\cdot\tau=\frac{a\tau+b}{c\tau+d}.
\]
We will assume, without loss of generality, that $2\leq k\leq l\leq m\leq \infty$. 

\begin{exmp}
The group ${\rm PSL}_2(\mathbb{Z})$ is a triangle group of type $(0;2,3,\infty)$.  It is well known that the generators of ${\rm SL}_2(\mathbb{Z})$ can be taken to be
\[ T=\begin{pmatrix}
    1 & 1 \\
    0 & 1
  \end{pmatrix} \;\;\;,\;\;\; S=\begin{pmatrix}
    0 & -1 \\
    1 & 0
  \end{pmatrix}.\]
By setting $g_1=-S$, $g_2=-T^{-1}S$ and $g_3=T$ we have that
\[{\rm SL}_2(\mathbb{Z})=\gen{g_1,g_2,g_3\;:\;g_1^{2}=g_2^3=g_1g_2g_3=-I}.\]
\end{exmp}
We have that the Schwarz triangle function $J_{(k,l,m)}(t)$ for a hyperbolic triangle $\bigtriangleup(k,l,m)$ satisfies the Schwarzian equation $(\star)$ with the rational function $R_{J}$ given as in Example \ref{ExampleTriangle} and with $\alpha=k,\beta=l,\gamma=m$ (cf. \cite[Chapter 5]{AbloFoka}). Very importantly, the functions $J_{(k,l,m)}(t)$ are automorphic uniformizers for $\Gamma_{(k,l,m)}$:
\[J_{(k,l,m)}(g\cdot \tau)=J_{(k,l,m)}(\tau)\;\;\;\text{ for all }\;g\in\Gamma_{(k,l,m)}\text{ and } \tau\in\mathbb{H}.\]

The first two main results of \cite{CasFreNag} are as follows
\begin{fct}{\cite[Theorem 2.12]{CasFreNag}}\label{hyperbolic}
Assume that $\bigtriangleup(k,l,m)$ is a hyperbolic triangle, that is assume that $k,l,m\in\mathbb{N}\cup\{\infty\}$ and satisfy the relation $\frac{1}{k}+\frac{1}{l}+\frac{1}{m}<1$. Then the set defined by the ODE $(\star)$ (with $R_{\alpha,\beta,\gamma}$ as in Example \ref{ExampleTriangle} and setting $\alpha=k,\beta=l,\gamma=m$) is strongly minimal and so geometrically trivial.
\end{fct}

\begin{fct} \label{generalschwarzian}
Assume that the set defined by the Schwarzian equation ($\star$) is strongly minimal and let $K$ be any differential field extension of $\mathbb{C}$. Then for any distinct solutions $y_1,\ldots,y_n$ of ($\star$) not in $K^{alg}$, if there is an algebraic relation between
\[y_1,y'_1,y''_1\cdots,y_n,y'_n,y''_n\]
over $K$, then there is a polynomial $P\in\mathbb{C}[X_1,X_2]$ and some $i<j$ such that
\[P(y_i,y_j)=0\] 
\end{fct}
This fact is obtained by putting together several results from \cite{CasFreNag}: If the set defined by ($\star$) is strongly minimal, it follows from Fact \ref{geo} that it also is geometrically trivial. So for any distinct solutions $y_1,\ldots,y_n$ of ($\star$) not in $K^{alg}$, if there is an algebraic relation between $y_1,y'_1,y''_1\cdots,y_n,y'_n,y''_n$ over $K$, then for some $i<j$ there is an algebraic relation between $y_i,y'_i,y''_i,y_j,y'_j,y''_j$ over $K$ (indeed over $\mathbb{C}$). But then \cite[Theorem 5.10]{CasFreNag} gives the desired result (see also \cite[Remark 5.14]{CasFreNag}).

A natural diving lines among the triangle groups is the notion of  arithmeticity. This notion plays a central role when tackling functional transcendence questions. We now give some details (for more see \cite{Vigneras}).
 
Let $F$ be a totally real number field of degree $k+1$ and denote by $\mathcal{O}_F$ its ring of integers. Let $A$ be a quaternion algebra over $F$ that is ramified at exactly one infinite place, that is,
\[A\otimes_{\mathbb{Q}} \mathbb{R}\simeq M_{2}(\mathbb{R})\times \mathcal{H}^k\]
where $\mathcal{H}$ is Hamilton's quaternion algebra $\left(\frac{-1,-1}{\mathbb{R}}\right)$. Let $\rho$ be the unique embedding of $A$ into $M_{2}(\mathbb{R})$ and let $\mathcal{O}$ be an order in $A$, namely a finitely generated $\mathcal{O}_F$-module such that $\mathcal{O}\otimes_{\mathcal{O}_F} F\simeq A$. The image $\rho(\mathcal{O}^1)$ of the norm-one group of $\mathcal{O}$ under $\rho$ is a discrete subgroup of ${\rm SL}_2(\mathbb{R})$. We denote by $\Gamma(A,\mathcal{O})$ the projection in ${\rm PSL}_2(\mathbb{R})$ of the group $\rho(\mathcal{O}^1)$.
\begin{defn}
A triangle group $\Gamma_{(k,l,m)}$ is said to be arithmetic if it is commensurable with a group of the form $\Gamma(A,\mathcal{O})$.
\end{defn}
Recall that two subgroups $G$ and $H$ of ${\rm PSL}_2(\mathbb{R})$ are commensurable, denoted by $G\sim H$, if their intersection $G\cap H$ has finite index in both $G$ and $H$. We now explain how commensurability and arithmeticity give rise to polynomials that violates the algebraic independence of solutions. First let us recall some facts about commensurable Fuchsian groups.
\begin{fct}\cite[Proposition 1.30 and 1.31]{Shi}\label{commenFact}
Let $\Gamma$ and $\Gamma_1$ be two commensurable Fuchsian groups, i.e., assume $\Gamma\sim\Gamma_1$. Then 
\begin{enumerate}
    \item $\Gamma$ and $\Gamma_1$ have the same set of cusps; and
    \item$\Gamma$ is of the first kind if and only if $\Gamma_1$ is of the first kind.
\end{enumerate}
\end{fct}

Let $\Gamma=\Gamma_{(k,l,m)}$ be a triangle group and let $\text{Comm}(\Gamma)$ be the commensurator of $\Gamma$, namely
\[\text{Comm}(\Gamma)=\{g\in {\rm PSL}_2(\mathbb{R})\;:\; g\Gamma g^{-1}\sim\Gamma\}.\]
If $g\in\text{Comm}(\Gamma)\setminus \Gamma$ then by Fact \ref{commenFact}, the intersection $\Gamma_g= g\Gamma g^{-1}\cap\Gamma$ is a Fuchsian group of the first kind with the same set of cusps as $\Gamma$. Since the functions $J_{(k,l,m)}(t)$ and $J_{(k,l,m)}(g^{-1}t)$ are respective automorphic uniformizers for $\Gamma$ and $g\Gamma g^{-1}$, we have that they also are automorphic functions for $\Gamma_g$. 

A classical theorem of Poincar\'e (cf. \cite[Chap. 5 Section 6]{Lehner}) states that any two automorphic functions for a Fuchsian group are algebraically dependent over $\mathbb{C}$. So there is a polynomial $\Phi_g\in\mathbb{C}[X,Y]$, such that $\Phi_g(J_{(k,l,m)}(t),J_{(k,l,m)}(gt))=0$. We call such polynomial a $\Gamma_{(k,l,m)}$-special polynomial and say that the uniformizers are in $\text{Comm}(\Gamma)$-correspondence. The following result of Margulis gives a characterization of arithmeticity in terms of $\Gamma_{(k,l,m)}$-special polynomials. 

\begin{fct}\cite{Marg}\label{ArithmeticityTHM}
The group $\Gamma_{(k,l,m)}$ is arithmetic if and only if $\Gamma_{(k,l,m)}$ has infinite index in $\text{Comm}(\Gamma)$ and so there are infinitely many $\Gamma_{(k,l,m)}$-special polynomials.
\end{fct}
This is a special case of a more general result of Margulis. For a proof we refer the reader to \cite[Chapter 6]{Zimmer}.
We can now state the next two main result from \cite{CasFreNag}. 
\begin{fct}{\cite[Theorem 2.13]{CasFreNag}}\label{IndArithmetic}
Suppose that $\Gamma=\Gamma_{(k,l,m)}$ is arithmetic and suppose that $j_{\Gamma}(g_{1}t),...,j_{\Gamma}(g_{n}t)$ are distinct solutions of the Schwarzian equation $(\star)$ that are pairwise not in $\text{Comm}(\Gamma)$-correspondence. Then the $3n$ functions $$j_{\Gamma}(g_{1}t),j'_{\Gamma}(g_{1}t),j''_{\Gamma}(g_{1}t),\ldots,j_{\Gamma}(g_{n}t),j'_{\Gamma}(g_{n}t),j''_{\Gamma}(g_{n}t)$$ are algebraically independent over ${\mathbb C}(t)$. 
\end{fct}

\begin{fct}{\cite[Theorem 2.14]{CasFreNag}}\label{IndNonArithmetic}
Suppose that $\Gamma=\Gamma_{(k,l,m)}$ is non-arithmetic. Then there is a $s\in\mathbb{N}$ such that if $j_{\Gamma}(g_{1}t),...,j_{\Gamma}(g_{n}t)$ are distinct solutions of the Schwarzian equation $(\star)$ satisfying 
\[\text{tr.deg.}_{\mathbb{C}(t)}\mathbb{C}\gen{t,j_{\Gamma}(g_{1}t)\ldots,j_{\Gamma}(g_{n}t)}=3n,\]
then for all other solutions $j_{\Gamma}(gt)$, except for at most $n\cdot s$, 
\[\text{tr.deg.}_{\mathbb{C}(t)}\mathbb{C}\gen{t,j_{\Gamma}(g_{1}t)\ldots,j_{\Gamma}(g_{n}t),j_{\Gamma}(gt)}=3(n+1).\]
\end{fct}

In Section 4 we will explain how one can refine these results in the non-arithmetic case. Finally, we state the Ax-Lidemann-Weierstrass Theorem with derivatives for the triangle groups $\Gamma=\Gamma_{(k,l,m)}$. 

\begin{defn}
We say that $t_1, \ldots , t_n $ are $\Gamma$-geodesically independent if $t_i$ is nonconstant for $i =1, \ldots , n$ and there are no relations of the form $t_i = \gamma t_j$ for $i \neq j$, $i,j \in \{1, \ldots , n \}$ and $\gamma\in\text{Comm}(\Gamma)$.
\end{defn}

The term geodesically independent comes from Pila's work \cite{PilaAO}; it is related to the earlier notion of a totally geodesic subvariety of a Shimura variety studied by Moonen \cite{Moonen}.

\begin{fct}{\cite[Theorem 2.16]{CasFreNag}}\label{ALW}
Let $\m C(V)$ be an algebraic function field, where $V\subset \m A^{n}$ is an irreducible algebraic variety defined over $\m C$. Let
\[t_1,\ldots,t_n\in \m C(V)\]
take values in the upper half complex plane $\m H$ at some $P\in V$ and are $\Gamma$-geodesically independent. Then the $3n$-functions
\[J_{(k,l,m)}(t_1),J'_{(k,l,m)}(t_1),J''_{(k,l,m)}(t_1)\ldots,J_{(k,l,m)}(t_n),J'_{(k,l,m)}(t_n),J''_{(k,l,m)}(t_n)\]
(considered as functions on $V(\m C)$ locally near $P$) are algebraically independent over $\m C(V)$.
\end{fct}

\section{The generic Schwarz triangle equation and Belyi Surfaces} \label{GenericTri}
Throughout we assume that $(\mathcal{U},\partial)$ is a saturated differentially closed field of characteristic $0$ and that $\mathbb{C}$ is its field of constants. We work in the language $L_{\partial}=(0,1,+,\cdot,\partial)$ of differential rings. Recall by a generic Schwarzian triangle equations we mean the ODE ($\star$) with 
\[R_{\alpha,\beta,\gamma}(y)=\frac{1}{2}\left(\frac{1-{\beta}^{-2}}{y^2}+\frac{1-{\gamma}^{-2}}{(y-1)^2}+\frac{{\beta}^{-2}+{\gamma}^{-2}-{\alpha}^{-2}-1}{y(y-1)}\right).\]
and such that $\alpha,\beta,\gamma$ are are three complex numbers algebraically independent
over $\mathbb{Q}$ (see example \ref{ExampleTriangle}).

\subsection{Strong Minimality}
We now aim to prove the following result.
\begin{prop}\label{THM1}
The set defined by a generic Schwarz triangle equation ($\star$) is strongly minimal.
\end{prop}
\begin{proof}[Proof of Proposition \ref{THM1}] Assume that $\alpha,\beta,\gamma$ are algebraically independent over $\mathbb{Q}$. We denote by $C$ the field of constants generated by $\alpha,\beta,\gamma$ over $\mathbb{Q}$, that is $C=\mathbb{Q}(\alpha,\beta,\gamma)$. Let us denote by $X(\alpha,\beta,\gamma)$ the set defined by ODE $(\star)$. We write $\phi(y,u_1,u_2,u_3)$ for the $L_{\partial}$-formula such that $X(\alpha,\beta,\gamma)=\{y\in\mathcal{U}\;:\;\models\phi(y,\alpha,\beta,\gamma)\}$.

Of course $\phi(y,u_1,u_2,u_3)$ is obtained from the equation $(\star)$ with the added condition that $y'\neq0$, which is required once we clear the denominators. More precisely since the numerator and denominator of $R_{\alpha,\beta,\gamma}(y)$ are $r(y,\alpha,\beta,\gamma)=(1-{\beta}^{-2})(y-1)^2+(1-{\gamma}^{-2})y^2+({\beta}^{-2}+{\gamma}^{-2}-{\alpha}^{-2}-1)y(y-1)$ and $s(y)=2y^2(y-1)^2$ respectively, if we let
$$\varphi(y,u_1,u_2,u_3)=\;\left(s(y)y'y'''-\frac{3}{2}s(y)(y'')^2+r(y,u_1,u_2,u_3)(y'^4)=0\right)\wedge (y'\neq 0),$$
then $\phi(y,u_1,u_2,u_3):= \varphi(y,u_1,u_2,u_3)\wedge (u_1'=0)\wedge (u_2'=0)\wedge (u_3'=0)$. 

For contradiction, assume that $X(\alpha,\beta,\gamma)$ is not strongly minimal. Then by definition, there exists a differential field extension $K$ of $C$ and $z\in X(\alpha,\beta,\gamma)$, such that $tr.deg._KK\gen{z}=1\text{ or }2$. We can assume that\footnote{For example, $\overline{b}$ can be taken to be the coefficients appearing in the polynomial witnessing $tr.deg._KK\gen{z}=1\text{ or }2$.} $K=C\gen{\overline{b}}$ for some $\overline{b}\in\mathcal{U}^m$ and $m\in\mathbb{N}$. So it follows that there exists a differential polynomial $F\in K\{y\}$ of order $1$ or $2$, such that $F(z)=0$. We will use the fact that we have a true first-order $L_{\partial}$-sentence ($\theta(\alpha,\beta,\gamma)$ below) asserting that: there exists a finite tuple of parameters $\overline{b}$ from $\mathcal{U}$, such $X(\alpha,\beta,\gamma)$ ($\alpha,\beta,\gamma$ constants) has a subvariety given by the vanishing of a non-trivial differential polynomial of order 1 or 2 and with coefficients in $\mathbb{Q}\gen{\alpha,\beta,\gamma,\overline{b}}$.

We write $F(y)$ as $F(y,\alpha,\beta,\gamma,\overline{b})$ to emphasize that those parameters appear in $F(y)$ and let $P_i(\alpha,\beta,\gamma,\overline{b})$ (for $i=1,\ldots,r$ and some $r\in\mathbb{N}$) be the non-zero coefficients of $F(y)$ as a polynomial in $y,y',y''$. Here each $P_i\in K=C\gen{\overline{b}}$. Let us also write $\rho(u_1,u_2,u_3,\overline{v})$ for the $L_{\partial}$-formula such that $\rho(\alpha,\beta,\gamma,\overline{b})$ is the true $L_{\partial}$-sentence
\[\forall y\left(F(y,\alpha,\beta,\gamma,\overline{b})=0\rightarrow \phi(y,\alpha,\beta,\gamma)\right). \]
In other words, $\rho(u_1,u_2,u_3,\overline{v}):=\forall y\left(F(y,u_1,u_2,u_3,\overline{v})=0\rightarrow \phi(y,u_1,u_2,u_3)\right)$
and  $\mathcal{U}\models \rho(\alpha,\beta,\gamma,\overline{b})$.

\noindent{\bf Claim:} We have a $L_{\partial}$-formula $\theta(u_1,u_2,u_3)$ such that {\bf\em if} $\mathcal{U}\models\theta(\alpha_0,\beta_0,\gamma_0)$, then $\alpha_0,\beta_0,\gamma_0$ are constants and there exists an order 1 or 2 definable subset of $X(\alpha_0,\beta_0,\gamma_0)$ defined over $\mathbb{Q}\gen{\alpha_0,\beta_0,\gamma_0,\overline{c}}$ for some $\overline{c}\in\mathcal{U}^m$.
\begin{proof}[Proof of Claim]
The formula $\theta(u_1,u_2,u_3)$ is simply chosen so that $\theta(\alpha,\beta,\gamma)$ is the above-mentioned true $L_{\partial}$-sentence, namely
\[\exists \overline{v}\left(\rho(\alpha,\beta,\gamma,\overline{v})\wedge \bigwedge^r_{i=1}P_i(\alpha,\beta,\gamma,\overline{v})\neq0\right).\]
If $\mathcal{U}\models\theta(\alpha_0,\beta_0,\gamma_0)$, then the formula $(u_1'=0)\wedge (u_2'=0)\wedge(u_3'=0)$ gives that $\alpha_0,\beta_0,\gamma_0$ are constants and the second statement of the claim follows by construction.\end{proof}
So we have $\theta(\alpha,\beta,\gamma)$ is true in $\mathcal{U}$ and so we can apply Fact \ref{descent} with $V = \m A^3$ and $F=\mathbb{Q}$. We obtained $k,l,m\in\mathbb{N}$ such that $2\leq k\leq l\leq m$ and $\mathcal{U}\models\theta(k,l,m)$. By making our initial choice of $k$ large enough (say $k>6$) we can also ensure that $\frac{1}{k}+\frac{1}{l}+\frac{1}{m}<1$. But now by the above claim, there exists an order 1 or 2 definable subset of $X(k,l,m)$. This contradicts Fact \ref{hyperbolic}, namely that $X(k,l,m)$ is strongly minimal.
\end{proof}

\begin{rem}
By virtue of Fact \ref{thm:strongminimality} and Proposition \ref{riccati}, it follows that we have obtained a different proof of the fact that if $\alpha,\beta,\gamma \in\m C$ are algebraically independent over $\m Q$, then the Riccati equation
$$\frac{du}{dy} + u^2 + \frac{1}{2}R_{\alpha,\beta,\gamma}(y) = 0$$
has no solutions in $\m C(t)^{alg}$.
\end{rem}
In the previous proof, note that the only point at which we used the fact that $(\alpha,\beta,\gamma )$ are independent transcendental numbers was in the final paragraph while applying Fact \ref{descent}. Therefore, it is not hard to see that the proof works identically in the following slightly more general case: 

\begin{cor} \label{genericpt}
The set defined by ($\star$) with $(\alpha, \beta , \gamma)$ a generic point on an irreducible algebraic variety $V \subset \m A^3$ over $\m Q$ with a dense set of points in $\m (N _{>1})^3$ is strongly minimal. 
\end{cor}

By means of differential Galois theory we can also specify the kind of $\mathbb Q$-algebraic dependence relations that may appear between
en the parameters $(\alpha,\beta,\gamma)$ for non-generic and non-strongly minimal triangle equations \eqref{(*)}.

\begin{prop}\label{pr:kimura}
Let us assume that equation \eqref{(*)} with $R = R_{\alpha,\beta,\gamma}$
with complex parameters $(\alpha,\beta,\gamma)$ is not strongly minimal. One of the following holds:
\begin{enumerate}
    \item At least one of the four complex numbers,
    $\alpha^{-1} + \beta^{-1} + \gamma^{-1}$, 
    $-\alpha^{-1} + \beta^{-1} + \gamma^{-1}$,
    $\alpha^{-1} - \beta^{-1} + \gamma^{-1}$,
    $\alpha^{-1} + \beta^{-1} - \gamma^{-1}$ is an odd integer.
    \item The quantities $\alpha^{-1}$ or $-\alpha^{-1}$, $\beta^{-1}$ or $-\beta^{-1}$ and $\gamma^{-1}$ or $-\gamma^{-1}$ take, in an arbitrary order, values given in the following table:

\begin{center}
\renewcommand{\arraystretch}{1.2}
\begin{tabular}{|c|c|c|c|c|}
\hline 
  & $\pm \alpha^{-1}$ & $\pm\beta^{-1}$ & $\pm \gamma^{-1}$ & \\ 
\hline 
1  & $\frac{1}{2}+\ell$ & $\frac{1}{2}+m$ & arbitrary & \\ 
\hline 
2 & $\frac{1}{2} + \ell$ & $\frac{1}{2}+m$ & $\frac{1}{2}+ n$ & \\  
\hline
3 & $\frac{2}{3}+\ell$ & $\frac{1}{3}+m$ & $\frac{1}{4}+n$ & $\ell+m+n$ even \\ 
\hline
4 &  $\frac{1}{2}+\ell$ & $\frac{1}{3}+m$ & $\frac{1}{4}+n$ &  \\ 
\hline
5 &  $\frac{2}{3}+\ell$ & $\frac{1}{4}+m$ & $\frac{1}{4}+n$ & $\ell+m+n$ even  \\
\hline
6 &  $\frac{1}{2}+\ell$ & $\frac{1}{3}+m$ & $\frac{1}{5}+n$ &   \\ 
\hline
7 &  $\frac{2}{5}+\ell$ & $\frac{1}{3}+m$ & $\frac{1}{3}+n$ & $\ell+m+n$ even  \\ 
\hline
8 &  $\frac{2}{3}+\ell$ & $\frac{1}{5}+m$ & $\frac{1}{5}+n$ & $\ell+m+n$ even  \\ 
\hline
9 &  $\frac{1}{2}+\ell$ & $\frac{2}{5}+m$ & $\frac{1}{5}+n$ & $\ell+m+n$ even  \\ 
\hline
10 &  $\frac{3}{5}+\ell$ & $\frac{1}{3}+m$ & $\frac{1}{5}+n$ & $\ell+m+n$ even  \\ 
\hline
11 &  $\frac{2}{5}+\ell$ & $\frac{2}{5}+m$ & $\frac{2}{5}+n$ & $\ell+m+n$ even  \\ 
\hline
12 &  $\frac{2}{3}+\ell$ & $\frac{1}{3}+m$ & $\frac{1}{5}+n$ & $\ell+m+n$ even  \\ 
\hline
13 &  $\frac{4}{5}+\ell$ & $\frac{1}{5}+m$ & $\frac{1}{5}+n$ & $\ell+m+n$ even  \\ 
\hline
14 &  $\frac{1}{2}+\ell$ & $\frac{2}{5}+m$ & $\frac{1}{3}+n$ & $\ell+m+n$ even  \\ 
\hline
15 &  $\frac{3}{5}+\ell$ & $\frac{2}{5}+m$ & $\frac{1}{3}+n$ & $\ell+m+n$ even  \\ 
\hline
\end{tabular}
\end{center}
where $\ell$, $m$, $n$ stand for arbitrary integer numbers.
\end{enumerate}
\end{prop}

\begin{proof}

First, let us recall the criterion of strong minimality given in Fact \ref{thm:strongminimality}.
If Riccati equation \eqref{eq:riccati} has no algebraic solutions, then equation \eqref{(*)} is strongly minimal. On the other hand, Riccati equation \eqref{eq:riccati} has an algebraic solution if and only if linear equation \eqref{eq:linear} is Liouville integrable (see \cite[\S 1.3]{Kovacic}).  Linear equation \eqref{eq:linear} is a Fuchsian equation with singular regular singularities at $0$, $1$ and $\infty$. Let us recall that two second order linear differential equations are projectively equivalent if the quotients between its pairs of linearly independent solutions satisfy the same Schwarzian equation (see \cite[Prop. VIII.3.2 p. 211]{SaintG}). The logarithmic derivative of solutions of projectively equivalent equations satisfy the same Riccati equation. It follows inmediately that a second order differential quation is Liouville integrable if and only if any projectively equivalent equation is so. Finally, any Fuchsian equation with singularities at $0$, $1$, $\infty$ projectively equivalent to an hypergeometric equation, 
\begin{equation}\label{eq:HG}
t(1-t)y'' + (c-(a+b+1)t)y' - aby = 0
\end{equation}
with the same exponent differences \cite[IX.2 p. 240]{SaintG}.
A direct computation of the exponent differences for \eqref{eq:linear} gives $\beta^{-1}$, $\gamma^{-1}$ and $\alpha^{-1}$ at $0$, $1$ and $\infty$ respectively. 
The exponent differences of the hypergeometric equation \eqref{eq:HG} are $1-c$, $c-a-b$ and $a-b$ at $0$, $1$ and $\infty$. Thus we obtain equations,
$$1-c = \beta^{-1}, \quad c-a-b = \gamma^{-1}, \quad a-b = \alpha^{-1}$$
(cf. \cite{Yoshida} page 68) and therefore we have, 
$$a = \frac{1}{2}(1+\alpha^{-1}-\beta^{-1}-\gamma^{-1}), \quad b = \frac{1}{2}(1-\alpha^{-1}-\beta^{-1}-\gamma^{-1}),\quad c = 1 -\beta^{-1}.$$
the parameters of an hypergeometric equation which is projectively equivalent to \eqref{eq:linear}. Finally, Liouville integrable hypergeometric equations \eqref{eq:HG} are completely classified by their exponent differences which, in this case, are $\alpha^{-1}$, $\beta^{-1}$, $\gamma^{-1}$. By application of Theorem I in \cite{Kimura} we obtain the desired result.
\end{proof}

\begin{rem}
Possibilities (1) and (2) in the statement of Proposition \ref{pr:kimura} correspond to different reductions of the Galois group. Case (1) is satisfied whenever the Galois group is triangularizable. Case (2) line 1, corresponds to a Galois group contained in the infinite dihedral group. Cases (2) lines 2 to 15, correspond to different realizations of symmetry groups of platonic solids. Hence, they are not completely exclusive one of the other: for instance, a diagonal group is triangularizable and also contained in the infinite dihedral group.
\end{rem}

Now that we have proved strong minimality in various cases, we aim to understand the existence of possible algebraic relations between solutions of the given equation. We are only able to do so in the generic case and leave other cases for future work. By Proposition \ref{THM1}, we have that the conclusion of Theorem \ref{generalschwarzian} holds for a generic Schwarz triangle equation ($\star$), and our next step will be to prove that no such polynomial as in Theorem \ref{generalschwarzian} exists. That is, there are no algebraic relations between solutions of a generic Schwarzian equation. Our argument exploits the fact that arithmetic triangle groups are rare.
\begin{fct}{\cite{Take}}
Up to ${\rm PSL}_2(\mathbb{R})$-conjugation, there are finitely many arithmetic triangle groups; 76 cocompact and 9 non-cocompact. A complete (finite) list of triples $(k,l,m)$ with $2\leq k\leq l\leq m\leq\infty$ such that $\Gamma_{(k,l,m)}$ is arithmetic can be found in \cite[Theorem 3]{Take}.
\end{fct}

So most triangle groups are non-arithmetic. We will use this to our advantage. We also need a finer analysis of the non-arithmetic groups, especially those groups which are equal to their commensurators (and thus have no associated special polynomials). 
\begin{defn}
A Fuchsian group $\Gamma$ is {\em maximal} if $\Gamma$ cannot be properly embedded, with finite index, in any Fuchsian group.
\end{defn}

\begin{fct}\label{trianglefct} Let $\Gamma=\Gamma_{(k,l,m)}$ be a triangle group. The following holds:
\begin{enumerate}
    \item Any Fuchsian group containing $\Gamma$ is itself a triangle group. 
    \item If $\Gamma$ is non-arithimetic, then $\text{Comm}(\Gamma)$ is a triangle group. Indeed, it is the unique maximal triangle group containing $\Gamma$.
\end{enumerate}
\end{fct}
\begin{proof}
The first assertion follows from Theorem \cite[Theorem 1]{Green} (see also \cite[Proposition 1 and Section 6]{Sing}. For (2), if $\Gamma$ is non-arthimetic we have, by Fact \ref{ArithmeticityTHM}, that $\Gamma$ has finite index in $\text{Comm}(\Gamma)$. From this we get that $\text{Comm}(\Gamma)$ is a Fuchsian group. Now from (1) it follows that $\text{Comm}(\Gamma)$ is a triangle group. Finally, if $\Gamma_1$ is a maximal Fuchsian group containing $\Gamma$, then it is not hard to see using $[\Gamma_1:\Gamma]<\infty$, that $g\Gamma g^{-1}\sim\Gamma$ for any $g\in\Gamma_1$. So $\Gamma_1\subseteq\text{Comm}(\Gamma)$ and the statement follows.
\end{proof}

So from Fact \ref{trianglefct} a non-arithmetic triangle group $\Gamma$ is  maximal if and only if $\Gamma=\text{Comm}(\Gamma)$. In this case there are no $\Gamma$-special polynomials. We have a precise description of when this occurs:
\begin{fct}\cite[Theorem 2]{Sing}\label{nonmaximal}
A triangle group $\Gamma_{(k,l,m)}$ is maximal if and only if $(k,l,m)$ is not of the form 
\begin{enumerate}
    \item $(2, l, 2l)$
    \item $(3, l, 3l)$
    \item $(k, l, l)$
\end{enumerate} 
with $k,l,m\in\mathbb{N}\cup\{\infty\}$ not necessarily in increasing order.
\end{fct}
We can now give a finer version of some results in \cite{CasFreNag}. We will denote by $\mathcal{M}$ the set of triples of natural numbers that has form given in Fact \ref{nonmaximal}.
\begin{thm}\label{nonarith}
Let $\Gamma=\Gamma_{(k,l,m)}$ be a non-arithmetic triangle group.
\begin{enumerate}
    \item If $(k,l,m)\in \mathcal{M}$, then there is $s\in\mathbb{N}_{>1}$ such that if $y_1,\ldots,y_n$ are distinct solutions of the Schwarzian equation $(\star)$ (with $\alpha=k,\beta=l,\gamma=m$) satisfying 
\[\text{tr.deg.}_{\mathbb{C}}\mathbb{C}\gen{y_1,\ldots,y_n}=3n,\]
then for all other solutions $y$, except for at most $n\cdot s$, 
\[\text{tr.deg.}_{\mathbb{C}}\mathbb{C}\gen{y_1,\ldots,y_n,y}=3(n+1).\]
\item If $(k,l,m)\not\in \mathcal{M}$, then for any distinct solutions $y_1,\ldots,y_n$ of the Schwarzian equation $(\star)$ (with $\alpha=k,\beta=l,\gamma=m$) we have that
\[\text{tr.deg.}_{\mathbb{C}}\mathbb{C}\gen{y_1,\ldots,y_n}=3n,\]
\end{enumerate}
\end{thm}
\begin{proof}
As mentioned above this is basically just a refinement of Theorem 2.14 in \cite{CasFreNag}. The ideas of the proof are as follows. First for triangle groups, Theorem \ref{generalschwarzian} is more precise. By lemmas 5.15 and 5.16 of \cite{CasFreNag}, the polynomials (in $\mathbb{C}[X_1,X_2]$) that can witness algebraic dependencies among solutions must all be $\Gamma$-special. 

Now, if $(k,l,m)\in \mathcal{M}$, then $\Gamma$ is properly contained in $\text{Comm}(\Gamma)$. But since $\Gamma$ is non-arithmetic, $\Gamma$ has finite index (say $s>1$) in $\text{Comm}(\Gamma)$. From this we get the desired result. On the other hand, if $(k,l,m)\not\in \mathcal{M}$, then $\Gamma=\text{Comm}(\Gamma)$ and so there all no $\Gamma$-special polynomials.
\end{proof}
We are now ready to prove our next theorem.
\begin{thm}\label{THM2}
The set defined by a generic Schwarz triangle equation ($\star$) is strictly disintegrated. That is, if $K$ is any differential field extension of $\mathbb{C}$ and $y_1,\ldots,y_n$ are distinct solutions that are not algebraic over $K$, then $$tr.deg._KK(y_1,y'_1,y''_1,\ldots,y_n,y'_n,y''_n)=3n$$.
\end{thm}
We will again make use of Fact \ref{descent}. In what follows $\mathcal{W}$ will denote the union of $\mathcal{M}$ with the finite set consisting of triples of natural number corresponding to arithmetic triangle groups.
\begin{proof}[Proof of Theorem \ref{THM2}] 
We begin with the same conventions as in the proof of Proposition \ref{THM1}. Let $\alpha,\beta,\gamma$ be algebraically independent over $\mathbb{Q}$ and denote by $C$ the field of constants generated by $\alpha,\beta,\gamma$ over $\mathbb{Q}$. We write $X(\alpha,\beta,\gamma)$ for the set defined by ODE $(\star)$ and assume $X(\alpha,\beta,\gamma)=\{y\in\mathcal{U}\;:\;\models\phi(y,\alpha,\beta,\gamma)\}$ for some $L_{\partial}$-formula $\phi(y,u_1,u_2,u_3)$. By Proposition \ref{THM1}, we have that $X(\alpha,\beta,\gamma)$ is strongly minimal.

Let $y_1,\ldots,y_n$ be distinct elements of $X(\alpha,\beta,\gamma)$ and for contradiction assume that there is an algebraic relation between
\[y_1,y'_1,y''_1\cdots,y_n,y'_n,y''_n.\]
Using Theorem \ref{generalschwarzian}, we have a polynomial $P\in\mathbb{C}[x,y]$ and some $i<j$ such that $P(y_i,y_j)=0$. Let us write $P(x,y)$ as $P(x,y,\alpha,\beta,\gamma,\overline{b})$ where $\overline{b}$ is a tuple of complex numbers all distinct from $\alpha,\beta,\gamma$. We let $\rho(u_1,u_2,u_3)$ be the $L_{\partial}$-formula such that $\rho(\alpha,\beta,\gamma)$ is the true $L_{\partial}$-sentence
\[\exists\overline{v}\exists x\exists y\left(\overline{v}'=\overline{0}\wedge x\neq y\wedge\phi(x,\alpha,\beta,\gamma)\wedge\phi(y,\alpha,\beta,\gamma)\wedge P(x,y,\alpha,\beta,\gamma,\overline{v})=0\right).\]
We can now use Fact \ref{descent} with $V = \m A^3$ and and $F=\mathbb{Q}$ to obtain a triple $(k,l,m)$ of natural number such that $(k,l,m)\not\in\mathcal{W}$, $2<k<l<m$, $\frac{1}{k}+\frac{1}{l}+\frac{1}{m}<1$ and $\mathcal{U}\models\rho(k,l,m)$. But this means that there is a polynomial in $\mathbb{C}[x,y]$ which vanishes on two distinct element of $X(k,l,m)$. This contradicts case (2) of Theorem \ref{nonarith}.
\end{proof}
Similar to Proposition \ref{THM1}, the proof of the previous result applies in slightly more generality, so similarly to Corollary \ref{genericpt}, we next sketch the most general version of Theorem \ref{THM2} which can be established with the methods we used above. 

\begin{cor} \label{genericpt_dis}
Suppose that $(\alpha, \beta, \gamma)$ is the generic point on an irreducible algebraic variety $V \subset \m A^3$ over $\m Q$ of dimension at least one such that $V$ has a dense set of points with coordinates in $\m N^3.$ Work with coordinates $(x,y,z).$ Assume further that
\begin{enumerate} 
\item $V$ is not the curve given by $x=2, z=2y$. 
\item $V$ is not the curve given by $x=3, z=3y$. 
\item $V$ is not contained in surface $y=z$. 
\end{enumerate} 
The set defined by the Schwarz triangle equation ($\star$) is strictly disintegrated. That is, if $K$ is any differential field extension of $\mathbb{C}$ and $y_1,\ldots,y_n$ are distinct solutions to ($\star$) with parameters $(\alpha, \beta, \gamma)$ that are not algebraic over $K$, then
$$tr.deg._KK(y_1,y'_1,y''_1,\ldots,y_n,y'_n,y''_n)=3n$$
\end{cor}

\begin{proof} The argument used in the proof of Theorem \ref{THM2} applies directly to the triple $(\alpha, \beta , \gamma)$ whenever  it has the property that
\begin{itemize} 
\item The Zariski closure of $(\alpha, \beta , \gamma)$ over $\m Q$ contains a dense set of points of $\m N^3$ which correspond to maximal triangle groups. 
\end{itemize} 

Work with coordinates $(x,y,z)$ in $\m A^3$ in what follows. So, when $V$ is any algebraic surface with dense $\m N$-points and $V$ is not given by $y=z$, the result follows. Suppose that $V$ is any algebraic curve with dense $\m N$-points such that none of the following hold: 
\begin{itemize} 
\item $V$ is given by $x=2, z=2y$.
\item $V$ is given by $x=3, z=3y$
\item $V$ lies on the surface $y=z$. 
\end{itemize} 
As long as none of the three conditions holds we have infinitely many points of $\m N^3$ corresponding to maximal triangle groups in the Zariski closure of $(\alpha, \beta, \gamma)$ over $\m Q$, and thus the argument of the proof of Theorem \ref{THM2} applies to yield our result. 
\end{proof}

\subsection{Orthogonality} \label{ortho}
We will now study the possible algebraic relations between solutions of two generic Schwarzian equations. We will show that the definable sets are orthogonal:

\begin{defn} Let $\mathcal X$ and $\mathcal Y$ be two strongly minimal sets both defined over some differential field $K$.
\begin{enumerate}
    \item $\mathcal X$ and $\mathcal Y$ are {\em nonorthogonal} if there is some definable (possibly with additional parameter) relation $\mathcal R\subset \mathcal X\times \mathcal Y$ such that the images of the projections of $\mathcal R$ to $\mathcal X$ and $\mathcal Y$ respectively are infinite and these projections are finite-to-one.
    \item $\mathcal X$ and $\mathcal Y$ are {\em  non weakly orthogonal} if they are nonorthogonal, that is there is an infinite finite-to-finite relation $\mathcal R\subseteq \mathcal X\times \mathcal Y$, and the formula defining $\mathcal R$ can be chosen to be over $K^{alg}$.
\end{enumerate}
\end{defn}

\begin{rem}\label{Remarkalg} Suppose $\mathcal X$ and $\mathcal Y$  (as above) are nonorthogonal and that the relation $\mathcal R\subset \mathcal X\times \mathcal Y$ witnessing nonorthogonality is defined over some differential field $F$ extending $K$. Then by definition for any $x\in \mathcal X\setminus F^{alg}$ there exist $y\in \mathcal Y\setminus F^{alg}$ such that $(x,y)\in \mathcal R$. In that case $F\gen{x}^{alg}=F\gen{y}^{alg}$, that is $x,y$ and derivatives are algebraically dependent over $F$.
\end{rem}
We will need the following important fact. We restrict ourselves to strictly disintegrated strongly minimal sets as this is all we need for the Schwarzian equations. We direct the reader to \cite[Corollary 2.5.5]{GST} for the more general context.

\begin{fct}\label{weaklyort}
Let $\mathcal X$ and $\mathcal Y$ be strongly minimal sets both defined over some differential field $K$. Assume further that they are both strictly disintegrated. If $\mathcal X$ and $\mathcal Y$ are nonorthogonal, then they are non weakly orthogonal.
\end{fct}
So by Theorem \ref{THM2}, we see that if the solution sets of two generic Schwarzian equations are nonorthogonal, then they are non weakly orthogonal. As with the proofs of \ref{THM1} and \ref{THM2} our strategy is to make a ``descent'' argument to the triangle groups. We review the relevant results in \cite{CasFreNag}.

Let $\Gamma_1$ and $\Gamma_2$ be two triangle group. We say that $\Gamma_1$ is commensurable with $\Gamma_2$ in {\it wide sense} if $\Gamma_1$ is commensurable to some conjugate of $\Gamma_2$. In particular, if $\Gamma_1$ is commensurable with $\Gamma_2$ in wide sense, then $\text{Comm}(\Gamma_1)$ is conjugate to $\text{Comm}(\Gamma_2)$.

\begin{rem}\label{differentType}
Suppose that $\Gamma_1=\Gamma_{(k_1,l_1,m_1)}$ and $\Gamma_2=\Gamma_{(k_2,l_2,m_2)}$ are two distinct maximal non-arithmetic triangle groups (that is assume $\text{Comm}(\Gamma_1)=\Gamma_1$ and $\text{Comm}(\Gamma_2)=\Gamma_2$). We have that $\Gamma_1$ is not commensurable with $\Gamma_2$ in wide sense. This follows since $\Gamma_1$ is not conjugate to $\Gamma_2$ - the two group not being of same type.
\end{rem}

\begin{fct}{\cite[Theorem 6.5]{CasFreNag}}\label{TriangleOrtho} Suppose that $\Gamma_{(k_1,l_1,m_1)}$ and $\Gamma_{(k_2,l_2,m_2)}$ are two triangle groups that are not commensurable in wide sense. Then the sets defined by the two Schwarzian equations ($\star$) (with parameters $(k_1,l_1,m_1)$ and $(k_2,l_2,m_2)$ respectively) are orthogonal.
\end{fct}
\begin{proof}
This is simply Theorem 6.5 of \cite{CasFreNag} restricted to the case of two triangle groups.
\end{proof}
This is all we need to prove the desired result. 

\begin{prop}
Let $\alpha_1,\beta_1,\gamma_1,\alpha_2,\beta_2,\gamma_2\in\mathbb{C}$ be algebraically independent over $\mathbb{Q}$. Let $\mathcal X(\alpha_1,\beta_1,\gamma_1)$ and $\mathcal X(\alpha_2,\beta_2,\gamma_2)$ be the set defined by the two generic Schwarzian equations ($\star$) (with parameters $(\alpha_1,\beta_1,\gamma_1)$ and $(\alpha_2,\beta_2,\gamma_2)$ respectively). Then $\mathcal X(\alpha_1,\beta_1,\gamma_1)$ is orthogonal to $\mathcal X(\alpha_2,\beta_2,\gamma_2)$.
\end{prop}
\begin{proof}
As before, $\mathcal{W}$ will denote the union of $\mathcal{M}$ - the set of triples of natural numbers that has form given in Fact \ref{nonmaximal} - with the finite set consisting of triples of natural number corresponding to arithmetic triangle groups. 

For contradiction, if $\mathcal X_1=\mathcal X(\alpha_1,\beta_1,\gamma_1)$ is nonorthogonal to $\mathcal X_2=\mathcal X(\alpha_2,\beta_2,\gamma_2)$, then there is a definable finite-to-finite relation $\mathcal R\subset \mathcal X_1\times \mathcal X_2$ between the two sets and we can assume that the relation is defined over $\m Q(\alpha_1,\beta_1,\gamma_1,\alpha_2,\beta_2,\gamma_2)^{alg}$. Let $\sigma(u_1,v_1,w_1,u_2,v_2,w_2)$ be the $L_{\partial}$-formula such that $\sigma(\alpha_1,\beta_1,\gamma_1,\alpha_2,\beta_2,\gamma_2)$ is the true $L_{\partial}$-sentence stating that $\mathcal R\subset \mathcal X_1\times \mathcal X_2$ is a definable finite-to-finite relation. 

We can now use Fact \ref{descent} with $V = \m A^3$ and $F=\mathbb{Q}(\alpha_2,\beta_2,\gamma_2)$ to specialize $(\alpha_1,\beta_1,\gamma_1)$ and get a triple of integers $(k_1,l_1,m_1)\not\in\mathcal{W}$ such that $2<k_1<l_1<m_1$, $\frac{1}{k_1}+\frac{1}{l_1}+\frac{1}{m_1}<1$ and $\mathcal{U}\models\sigma(k_1,l_1,m_1,\alpha_2,\beta_2,\gamma_2)$. Notice that $\Gamma_{(k_1,l_1,m_1)}$ is a maximal non-arithmetic triangle group. 

Now we again apply Fact \ref{descent} with $V = \m A^3$ and $F=\mathbb{Q}$ - this time to specialize $(\alpha_2,\beta_2,\gamma_2)$ - and choose a triple of integers $(k_2,l_2,m_2)\not\in\mathcal{W}\cup\{(k_1,l_1,m_1)\}$ such that $2<k_2<l_2<m_2$, $\frac{1}{k_2}+\frac{1}{l_2}+\frac{1}{m_2}<1$ and $\mathcal{U}\models\sigma(k_1,l_1,m_1,k_2,l_2,m_2)$. This time we have a maximal non-arithmetic triangle group $\Gamma_{(k_2,l_2,m_2)}$ which is distinct from $\Gamma_{(k_1,l_1,m_1)}$.

But this means that there is a definable relation between $\mathcal X(k_1,l_1,m_1)$ and $\mathcal X(k_2,l_2,m_2)$, that is they are nonorthogonal. But the triples where chosen so that $\Gamma_{(k_1,l_1,m_1)}$ is not commensurable with $\Gamma_{(k_2,l_2,m_2)}$ in wide sense (see Remark \ref{differentType}). This contradicts Fact \ref{TriangleOrtho} above.
\end{proof}

\subsection{Non-zero Fibers of generic Schwarzian triangle equations} In this subsection we consider the differential operator 
\begin{eqnarray} \label{chieqn} \chi _{\alpha,\beta,\gamma, \frac{d}{dt}} (y) := S_{t}(y) +(y')^2 R_{\triangle}(y)\end{eqnarray}
and for $a\in\mathcal{U}$, study equations of the form  $$\chi _{\alpha,\beta,\gamma, \frac{d}{dt}}  (y) = a.$$ We call such equations the fibers of the Schwarzian triangle equations. For Fuchsian groups of the first kind and genus zero we have a complete description of the structure of the set of solutions of the fibers of the corresponding Schwarzian equations. The main result, stated in the case of the triangle groups, is as follows
\begin{fct}{\cite[Theorem 6.2]{CasFreNag}}\label{groupfiber}
Let $a\in\mathcal{U}$. Assumme that $\bigtriangleup(k,l,m)$ is a hyperbolic triangle. Then the set defined $\chi _{\alpha,\beta,\gamma, \frac{d}{dt}}  (y) = a$ (with $\alpha=k,\beta=l,\gamma=m$) is strongly minimal and geometrically trivial. Furthermore if $a_1, \ldots , a_n\in\mathcal{U}$ satisfy $\chi _{\alpha,\beta,\gamma, \frac{d}{dt}} (a_i) = a$ and are dependent\footnote{That is, there is a differential field extension $K$ of $\mathbb{Q}\gen{a}$ such that $tr.deg_KK\gen{a_1, \ldots , a_n}<3n$}, then there exist $i<j\leq n$ and a $\Gamma$-special polynomial $P$, such that $P(a_i, a_j) = 0$.
\end{fct}

\begin{rem}\label{nonarithmeticfiber} 
Note that if $(k,l,m)\not\in\mathcal{W}$, where $\mathcal{W}$ still denotes the union of $\mathcal{M}$ with the finite set of triples for arithmetic triangle groups, then $\Gamma$-special polynomials do not exists. As such Theorem \ref{groupfiber} tell us that if in addition $(k,l,m)\not\in\mathcal{W}$, then the set defined $\chi _{\alpha,\beta,\gamma, \frac{d}{dt}}  (y) = a$ is strictly disintegrated.
\end{rem}

It is now clear from the work in \cite{CasFreNag} and \cite{FreSca} that, in the case of a Fuchsian group $\Gamma$, if one is able to show that the corresponding Schwarzian equation is strongly minimal, then one can conclude that any non-zero fiber of the equation is strongly minimal. This follows since the solutions of the non-zero fibers can be written in terms of the automorphic uniformizer $j_{\Gamma}$ of $\Gamma$. Using this fact and the chain rule, one can then reduce the problem to determining strong minimality of the zero fiber (see \cite[Section 6]{CasFreNag} for more details). 

Outside the context of Fuchsian groups, one cannot write all solutions of non-zero fibers in terms of some solution of the zero fiber. Nevertheless, we are able to use the same techniques as in the previous subsections to study non-zero fibers of the generic Schwarzian triangle equations.
\begin{thm}
Let $a\in\mathcal{U}$ be non-zero. Assume that $\alpha$, $\beta$, $\gamma$ are algebraically independent over $\mathbb{C}$. Then the set defined by $\chi _{\alpha,\beta,\gamma, \frac{d}{dt}}  (y) = a$ is strongly minimal and strictly disintegrated.
\end{thm}
\begin{proof} 
Assume that $\alpha,\beta,\gamma$ are algebraically independent over $\mathbb{Q}$. As before, denote by $C$ the field of constants generated by $\alpha,\beta,\gamma$ over $\mathbb{Q}$, that is $C=\mathbb{Q}(\alpha,\beta,\gamma)$. Let us denote by $X(\alpha,\beta,\gamma,a)$ the set defined by $\chi _{\alpha,\beta,\gamma, \frac{d}{dt}}  (y) = a$. 

The proof of strong minimality has some similarity to that of the proof of Proposition \ref{THM1}. So some details are omitted. For contradiction, assume that $X(\alpha,\beta,\gamma,a)$ is not strongly minimal. Then for some differential field $K=C\gen{a,\overline{b}}$, where $\overline{b}\in\mathcal{U}^m$ and some $z\in X(\alpha,\beta,\gamma,a)$, we have that $tr.deg._KK\gen{z}=1\text{ or }2$. So we have a differential polynomial $F\in K\{y\}$ of order $1$ or $2$, such that $F(z)=0$. We write $F(y)$ as $F(y,\alpha,\beta,\gamma,a,\overline{b})$ to emphasize that those parameters appear in $F(y)$ and write $P_i(\alpha,\beta,\gamma,a,\overline{b})$ (for $i=1,\ldots,r$) for the non-zero coefficients of $F(y)$. Let us also write $\rho(u_1,u_2,u_3,u,\overline{v})$ for the $L_{\partial}$-formula such that $\rho(\alpha,\beta,\gamma,a,\overline{b})$ is the true $L_{\partial}$-sentence
\[\forall y\left(F(y,\alpha,\beta,\gamma,a,\overline{b})=0\rightarrow y\in X(\alpha,\beta,\gamma,a)\right). \]
{\bf Claim:} We have a $L_{\partial}$-formula $\theta(u_1,u_2,u_3)$ such that if $\mathcal{U}\models\theta(\alpha_0,\beta_0,\gamma_0)$, then $\alpha_0,\beta_0,\gamma_0$ are constants and there exists $a_0\in\mathcal{U}$ and an order 1 or 2 definable subset of $X(\alpha_0,\beta_0,\gamma_0,a_0)$ defined over $\mathbb{Q}\gen{\alpha_0,\beta_0,\gamma_0,a_0,\overline{c}}$ for some $\overline{c}\in\mathcal{U}^m$.
\begin{proof}[Proof of Claim]
The formula $\theta(u_1,u_2,u_3)$ is simply chosen so that $\theta(\alpha,\beta,\gamma)$ is the true $L_{\partial}$-sentence
\[\exists u\exists \overline{v}\left(\rho(\alpha,\beta,\gamma,u,\overline{v})\wedge \bigwedge^r_{i=1}P_i(\alpha,\beta,\gamma,u,\overline{v})\neq0\right)\]
\end{proof}
We have that $\theta(\alpha,\beta,\gamma)$ is true in $\mathcal{U}$ and so we can apply Fact \ref{descent} with $V = \m A^3$ and $F=\mathbb{Q}$. We obtained $k,l,m\in\mathbb{N}$ such that $2\leq k\leq l\leq m$,  $\frac{1}{k}+\frac{1}{l}+\frac{1}{m}<1$ and $\mathcal{U}\models\theta(k,l,m)$. But now by the above claim, there is $a_0\in \mathcal{U}$ such that there exists an order 1 or 2 definable subset of $X(k,l,m,a_0)$. This contradicts Theorem \ref{groupfiber}.

Now to the proof of strict disintegratedness: For contradiction, assume that $X(\alpha,\beta,\gamma,a)$ is not strictly disintegrated. Then for some differential field $K=C\gen{a,\overline{b}}$, where $\overline{b}\in\mathcal{U}^m$ and some $z_1,\ldots,z_{n+1}\in X(\alpha,\beta,\gamma,a)$, we have that $tr.deg._KK\gen{z_1,\ldots,z_{n+1}}\neq3(n+1)$. By strong minimality we have that $z_{n+1}\in K\gen{\overline{z}}^{alg}$, where $\overline{z}=(z_1,\ldots,z_n)$. 

Let $\varphi(u,\overline{v},\alpha,\beta,\gamma,a,\overline{b})$ be the $L_{\partial}$-formula that witness this, i.e. $\mathcal{U}\models\varphi(z_{n+1},\overline{z},\alpha,\beta,\gamma,a,\overline{b})$ and for any $y_{n+1},\overline{y}$ such that $\mathcal{U}\models\varphi(y_{n+1},\overline{y},\alpha,\beta,\gamma,a,\overline{b})$, we have that $y_{n+1}\in K\gen{\overline{y}}^{alg}$. Note here that the variables $u$ and $\overline{v}$ are in the sorts $X(\alpha,\beta,\gamma,a)$.

Consider the $L_{\partial}$-formula $\theta(u_1,u_2,u_3)$ so that $\theta(\alpha,\beta,\gamma)$ is the true $L_{\partial}$-sentence
\[\exists u\exists \overline{v}\exists \overline{w}\exists x\left(\varphi(u,\overline{v},\alpha,\beta,\gamma,x,\overline{w})\right).\]
If $\mathcal{U}\models\theta(\alpha_0,\beta_0,\gamma_0)$, then $\alpha_0,\beta_0,\gamma_0$ are constants and there exists $a_0\in\mathcal{U}$ and $y_1,\ldots,y_{n+1}\in X(\alpha_0,\beta_0,\gamma_0,a_0)$ such that $y_1,\ldots,y_{n+1}$ are interalgebraic over $\mathbb{Q}\gen{\alpha_0,\beta_0,\gamma_0,a_0,\overline{c}}$ for some $\overline{c}\in\mathcal{U}^m$. But then if we apply Fact \ref{descent} with $V = \m A^3$, we obtain a triple $(k,l,m)$ of natural number such that $(k,l,m)\not\in\mathcal{W}$, $2<k<l<m$, $\frac{1}{k}+\frac{1}{l}+\frac{1}{m}<1$ and $\mathcal{U}\models\theta(k,l,m)$. But then this contradicts Fact \ref{groupfiber} (also see Remark \ref{nonarithmeticfiber}). 
\end{proof}

\subsection{Belyi Surfaces}
Let $\Gamma$ be a Fuchsian Group of the first kind and not necessarily of genus zero. If the compactification $\mathcal{C}_{\Gamma}$ of the quotient $\Gamma\setminus\mathbb{H}$ is defined over $\mathbb{Q}^{alg}$, then  $\mathcal{C}_{\Gamma}$ is called a Belyi surface. The following theorem, proved by Belyi \cite{Belyi} (in this form see for example \cite[Theorem 4.1]{Sing1}), will play an important role.

\begin{fct}\label{belyi} Let $\Gamma$ be a Fuchsian group of the first kind.
\begin{enumerate}
\item If $\Gamma$ is cocompact, then $\mathcal{C}_{\Gamma}$ is a Belyi surface if and only if $\Gamma$ is a finite index subgroup of a cocompact triangle group $\Gamma_{(k,l,m)}$.
\item If $\Gamma$ is not cocompact, then $\mathcal{C}_{\Gamma}$ is a Belyi surface if and only if one of the following holds
\begin{enumerate}
    \item[(i)] $\Gamma$ is a finite index subgroup of $\Gamma_{(2,3,\infty)}$; or
    \item[(ii)] $\Gamma$ is a finite index subgroup of $\Gamma_{(2,\infty,\infty)}$; or
    \item[(iii)] $\Gamma$ is a finite index subgroup of $\Gamma_{(\infty,\infty,\infty)}$.
\end{enumerate}
\end{enumerate}
\end{fct}

It follows from Fact \ref{belyi} that $\mathcal{C}_{\Gamma}$ is a Belyi surface if and only if there exists a morphism $f:\mathcal{C}_{\Gamma}\rightarrow\mathbb{P}^1$ that is ramified only over $0$, $1$ and $\infty$. Our main result in this section is 
\begin{thm}\label{THM3}
Let $\Gamma$ be a Fuchsian group of the first kind and assume that $\mathcal{C}_{\Gamma}$ is a Belyi surface. Then the set defined by the Schwarzian equation for $\Gamma$ is strongly minimal and geometrically trivial. Furthermore the Ax-Lindemann-Weierstrass Theorem holds for $\Gamma$.
\end{thm}

Recall that if $\Gamma$ is a Fuchsian group and $j_{\Gamma}$ a uniformizing function, then we say that the {\em Ax-Lindemann-Weierstrass Theorem (ALW) holds for $\Gamma$} if the following condition is proven to hold:
Let $\m C(V)$ be an algebraic function field, where $V\subset \m A^{n}$ is an irreducible algebraic variety defined over $\m C$. Let $t_1,\ldots,t_n\in \m C(V)$ take values in the upper half complex plane $\m H$ at some $P\in V$ and are $\Gamma$-geodesically independent. Then the $3n$-functions
\[j_{\Gamma}(t_1),j'_{\Gamma}(t_1),j''_{\Gamma}(t_1)\ldots,j_{\Gamma}(t_n),j'_{\Gamma}(t_n),j''_{\Gamma}(t_n)\]
(considered as functions on $V(\m C)$ locally near $P$) are algebraically independent over $\m C(V)$.

We have the following very general proposition.
\begin{prop}\label{FiniteIndex}
Let $\Gamma$ be a Fuchsian group of the first  kind and assume that $\Gamma_1$ is a finite index subgroup of $\Gamma$. The Ax-Lindemann-Weierstrass Theorem holds for $\Gamma$ if and only if it holds for $\Gamma_1$. 
\end{prop}
\begin{proof}
Let $j$ and $j_1$ be uniformizing functions for $\Gamma$ and $\Gamma_1$ respectively. Since $\Gamma_1<\Gamma$, we have that $j$ and $j_1$ are automorphic functions for $\Gamma_1$. So as in the discussion following Fact \ref{commenFact}, we have that $j$ and $j_1$ are algebraically dependent over $\mathbb{C}$. Furthermore notice that $\Gamma$ is commensurable with $\Gamma_1$ and so $\text{Comm}(\Gamma)=\text{Comm}(\Gamma_1)$. 

If $\Phi\in\mathbb{C}[X_1,X_2]$ is the non-zero polynomial that witnesses that $j$ and $j_1$ are algebraically dependent over $\mathbb{C}$, that is $\Phi(j(t),j_1(t))=0$, then for any $g\in {\rm PSL}_2(\mathbb{C})$, we have that $\Phi(j(gt),j_1(gt))=0$. The Ax-Lindemann-Weierstrass Theorem for $\Gamma$ completely describes the possible algebraic relations between of $j(t)$ and $j(gt)$. So using that $j(gt)$ and $j_1(gt)$ are algebraically dependent over $\mathbb{C}$, we have that for any $g_1,\ldots,g_n$ which lie in distinct cosets of $\text{Comm}(\Gamma)$, the functions $j_1(g_1t),\ldots,j_1(g_nt)$ (and derivatives) are algebraically independent over $\mathbb{C}$. Clearly, the above argument works with the roles of $j$ and $j_1$ interchanged.  From this the result follows.
\end{proof}
\begin{cor}
Let $G_1$ and $G_2$ be two Fuchsian groups of the first kind which are commensurable in the wide sense. Then the Ax-Lindemann-Weierstrass Theorem holds for $G_1$ if and only if it holds for $G_2$. 
\end{cor}
\begin{proof}
Recall that we say that $G_1$ and $G_2$ are commensurable in the wide sense if some conjugate of $G_1$ is commensurable to $G_2$. That is, $G_1$ and $G_2$ are commensurable in the wide sense if $G_1$ and $G_2$  have the property that for some conjugate of $G_1$, say $g G_1 g^{-1}$ for $g \in PSL_ ( \m R)$, 
$G=g G_1 g^{-1} \cap G_2$ is a finite index subgroup of both $g G_1 g^{-1}$ and $G_2$. Applying Proposition \ref{FiniteIndex}, first with $\Gamma_1=G$ and $\Gamma= g G_1 g^{-1}$ and then again with $\Gamma_1=G$ and $\Gamma=G_2$, the result can be seen to follow if one shows that the ALW holds for $G_1$ if and only if it follows for $g G_1 g^{-1}.$ This last equivalence is true because if $j_{G_1}$ is a uniformizer for $G_1$, then $j_{g G_1 g^{-1}} (t)  := j_ {G_1} (g^{-1} t)$ is a unformizer for $gG_1 g^{-1}$. 
\end{proof}
\begin{proof}[Proof of Theorem \ref{THM3}]
We fix a triangle group $\Gamma_{(k,l,m)}$ such that $\Gamma<\Gamma_{(k,l,m)}$ as in Fact \ref{belyi}. Since the uniformizers $j_{\Gamma}$ and $J_{(k,l,m)}$ are interalgebraic over $\mathbb{C}$, by Fact \ref{hyperbolic}, the \emph{type} of $j_{\Gamma}$ over $\mathbb{C}$ is strongly minimal.\footnote{The type of $j_{\Gamma}$ over $\mathbb{C}$ is a specific instance of the general model theoretic notion of a type. In this setting, the type $j_{\Gamma}$ over $\mathbb{C}$ is the set defined by the collection of differential polynomial equations over $\m C$ satisfied by $j_{\Gamma}$ along with the (possibly infinitely many) differential polynomial inequations satisfied by $j_{\Gamma}$ over $\m C$.} Let $V$ be the differential variety which is given by the closure of $j_{\Gamma}$ over $\mathbb{C}$ (given by the vanishing of the Schwarzian equation). The main theorem of Nishioka \cite{Nish} for the automorphic function $j_{\Gamma}$ gives that for any function $f$ which lies in $V$, $f$ satisfies no differential equation of order less than three. From this it follows that that the set defined by the Schwarzian equation for $\Gamma$ is strongly minimal. Finally, Proposition \ref{geo} gives geometric triviality and Fact \ref{ALW} and Proposition \ref{FiniteIndex} gives the Ax-Lindemann-Weierstrass Theorem for $\Gamma$.

\end{proof}
\section{Bi-algebraic curves of general Schwarzian equations in genus 0.} \label{bialgcurves}
 Let $Y_1$ and $Y_2$ be two hyperbolic algebraic curves over $\mathbb C$ of genus zero (diffeomorphic to $\mathbb C P^1 \setminus S$ with $\# S \geq 3$), $y_1$ and $y_2$ be affine coordinates on $Y_1$ and $Y_2$ and $R_1$ and $R_2$ be rational functions on $Y_1$ and $Y_2$ respectively. Consider the following differential equations over the differential field $\left(\mathbb C (t_1,t_2), \frac{\partial}{\partial t_1}, \frac{\partial}{\partial t_2}\right)$:

\begin{equation}
\label{R_1}
S_{t_1}(y_1) +  \left(\frac{\partial y_1}{\partial t_1}\right)^2 R_1(y_1) = 0 \quad ; \quad  \frac{\partial y_1}{\partial t_2} = 0
\end{equation}

\begin{equation}
\label{R_2}
S_{t_2}(y_2) +  \left(\frac{\partial y_2}{\partial t_2}\right)^2 R_2(y_2) = 0 \quad ; \quad \frac{\partial y_2}{\partial t_1}=0
\end{equation}
 We fix a solution $(J_1, J_2)$ with $J_1 : U_1 \to Y_1$ and $J_2 : U_2 \to Y_2$ holomorphic on some domains of $\CC$ and consider the map $J : U_1 \times U_2 \to Y_1 \times Y_2$ sending $(t_1,t_2)$ on $\left(J_1(t_1), J_2(t_2)\right)$. An algebraic curve $\mathcal C \subset \mathbb C^2$ is {\em bi-algebraic with respect to $J$} (or simply bi-algebraic) if the Zariski closure of $J(\mathcal C \cap (U_1\times U_2))$ is an  algebraic curve in $Y_1 \times Y_2$. This algebraic curve will be denoted by $J(\mathcal C)$.

\begin{rem} \label{above}
With the assumption above
\begin{enumerate}
\item If $\tilde y_1 = h(y_1)$ is a different affine coordinate on $\mathbb C P^1$ with $h \in {\rm PSL}_2(\mathbb C)$, then the relevant Schwarzian equation is given with $\tilde R_1 = R_1 \circ h^{-1} \left(\frac{\partial h^{-1}}{\partial t_1} \right)^2$ and the solution of this new equation is $\tilde J_1 = h \circ J_1$.

\item Being bi-algebraic is independent of the choice of the solution $(J_1,J_2)$ since any other solution will be of the form $(J_1\circ h_1, J_2 \circ h_2)$ with $(h_1, h_2) \in \rm{PSL}_2(\mathbb C) \times \rm{PSL}_2(\mathbb C)$.  
\end{enumerate}
\end{rem}

In section \ref{section2} we gave some properties of the Schwarzian equation $S_t(y) + \left(\frac{\partial y}{\partial t}\right)^2R(y)$ under the following hypothesis on $R$, called the \emph{Riccati hypothesis:}
\begin{center}
\emph{The equation $\frac{du}{dy} + u^2 +\frac{1}{2}R(y)=0$ has no solutions in $\mathbb C(y)^{alg}$.}
\end{center}

We will show that under the Riccati hypothesis, bi-algebraic curves are very simple; namely, they are graphs of homographies. Moreover if one controls the polar locus of the two rational functions $R_1$ and $R_2$ then the curve $J(\mathcal C)$ is a ${\rm Comm}(\Gamma_1)$-correspondence between Zariski opens subsets $Y_1^\star \subset Y_1$ and $Y_2^\star \subset Y_2$ for some Fuchsian group $\Gamma_1$ given by the image of $\pi_1(Y_1^\star) \subset {\rm PSL}_2(\mathbb C)$.

\begin{lem}\label{R1IFFR2}
With the above notation, if $\mathcal C \subset \mathbb C^2$ is a bi-algebraic curve, then $R_1$ satisfies the Riccati hypothesis if and only if $R_2$ does.
\end{lem}

\begin{proof}

From Proposition \ref{riccati}, since $R_1$ satisfies the Riccati hypothesis we have that the solution $J_1$ of the equation (\ref{R_1}) is such that $t_1, J_1, \frac{\partial J_1}{\partial t_1},\frac{\partial^2 J_1}{\partial t_1^2}$ are algebraically independent over $\CC$, i.e., $tr.deg._{\m C}\m C(t_1, J_1, \frac{\partial J_1}{\partial t_1},\frac{\partial^2 J_1}{\partial t_1^2})=4$. 

Let $\mathcal C$ be a bi-algebraic curve for $J = (J_1, J_2)$, where $J_2$ is the solution of the equation (\ref{R_2}). Also let the polynomial equations of the bi-algebraic curve be $P(t_1,t_2)=0$ in $\mathbb C^2$ and $Q(y_1, y_2)=0$ in $Y_1 \times Y_2$.

Then on $\mathcal C$:
\begin{itemize}
\item $t_1 \in \mathbb C(t_2)^{alg}$ implying $\frac{\partial t_2}{\partial t_1} \in \mathbb C (t_2)^{alg}$,
\item $J_1 \in \mathbb C(J_2)^{alg}$ implying $\frac{\partial J_1}{\partial t_2} \in \mathbb C (J_2, \frac{\partial J_2}{\partial t_2})^{alg}$,
\item $\frac{\partial J_1}{\partial t_1} = \frac{\partial J_1}{\partial t_2}\frac{\partial t_2}{\partial t_1} \in \mathbb C(t_2,J_2,\frac{\partial J_2}{\partial t_2})^{alg}$.
\end{itemize}
Using the chain rule (one more time), we obtain that $$\mathbb C\left(t_1, J_1, \frac{\partial J_1}{\partial t_1},\frac{\partial^2 J_1}{\partial t_1^2}\right)^{alg}\subseteq\mathbb C\left(t_2, J_2, \frac{\partial J_2}{\partial t_2},\frac{\partial^2 J_2}{\partial t_2^2}\right)^{alg}$$ and so $tr.deg._{\m C}\m C(t_2, J_2, \frac{\partial J_2}{\partial t_2},\frac{\partial^2 J_2}{\partial t_2^2})=4$. Using Proposition \ref{riccati} we get that $R_2$ satisfies the Riccati hypothesis. Interchanging the roles of $J_1$ and $J_2$ proves the lemma.
\end{proof}

\begin{exam}\label{exampleR1}
If $R_1$ is a constant then the Riccati equation has (one or) two constant solutions $\pm \sqrt{\frac{-1}{2}R_1}$. It does not satisfy the Riccati hypothesis.
\end{exam}

\begin{exam}\label{exampleR2}
If $R_1(y_1) = 2y_1$, the equation $\frac{du}{dy_1}+u^2+y_1=0$ is the Riccati equation attached to the Airy equation $\frac{d^2 z}{dy_1^2} + y_1 z =0$. The differential Galois group of the latter Airy equation is $\rm{SL}_2(\mathbb C)$ (cf. \cite[examples 4.29 and 6.21]{magid}). By proposition \ref{riccati}, the rational function $R_1$ satisfies the Riccati hypothesis.
\end{exam}

\begin{exam}
Using the examples \ref{exampleR1} and \ref{exampleR2}, and Lemma \ref{R1IFFR2} one gets: if $R_1$ is a constant and $R_2(y_2) = 2y_2$, then for any solutions $J_1$ and $J_2$ of equations \ref{R_1} and \ref{R_2} respectively, there are no bi-algebraic curves with respect to $J = (J_1,J_2)$ but the vertical and horizontal ones.
\end{exam}

Throughout, vertical and horizontal algebraic subvarieties of $\mathbb C^2$ will respectively mean varieties of the form $\{b\}\times \mathbb C$ and $\mathbb{C} \times \{b\}$ with $b \in \mathbb C$.

\begin{thm}\label{bialgebraic1}
Let $J=(J_1,J_2)$ be a solution such that both \ref{R_1} and \ref{R_2} satisfy the Riccati hypothesis.
If $\mathcal C$ is neither a vertical nor a horizontal bi-algebraic curve with respect to $J$, then $\mathcal C$ is the graph of an homography.
\end{thm}

\begin{proof}
Vertical and horizontal curves are clearly bi-algebraic. Assume that $\mathcal C$ is not vertical or horizontal. 

Consider the field $K_1 = \CC(t_1,y_1,y_1',y_1'')$ with the four derivations
\begin{itemize}
    \item $D_1 = \frac{\partial}{\partial t_1} + y_1'\frac{\partial}{\partial y_1}+ y_1'\frac{\partial}{\partial y_1'} +  \left( \frac{3}{2}\frac{y_1''^2}{y_1'} -(y_1')^3 R(y_1) \right)\frac{\partial}{\partial y_1''}$, is the derivation such that $K_1$ is the field generated by a generic solution of the equation \ref{R_1}.
\item $X_1= - \frac{\partial}{\partial t_1} $,
\item $H_1 = - t_1 \frac{\partial}{\partial t_1} + y_1'\frac{\partial}{\partial y_1'} + 2y_1''\frac{\partial}{\partial y_1''}$,
\item $Y_1 =- \frac{t_1^2}{2}\frac{\partial}{\partial t_1}+t_1y_1'\frac{\partial}{\partial y_1'} + (y_1' + 2t_1y_1'')\frac{\partial}{\partial y_1''} $.
\end{itemize}
The last three derivations are the action of the Lie algebra $\mathfrak{psl}_2(\mathbb C)$ seen as the Lie algebra of infinitesimal generators of the group of fractional linear maps : $\mathbb C\frac{d}{dt} +  \mathbb C t\frac{d}{dt}+ \mathbb C \frac{t^2}{2}\frac{d}{dt}$. 
Notice that the action of  $\mathbb C X_1 + \mathbb C H_1 + \mathbb C Y_1$  preserves $\CC(t_1) \subset K_1$ and the induced action is the action of $\mathfrak{psl}_2(\mathbb C)$.
 One has $[D_1,X_1]=0$, $[D_1,H_1] =-D_1$ and $[D_1,Y_1]=-t_1D_1$.  

We also define $K_2$, $D_2$, $X_2$, $H_2$ and $Y_2$ analogously.

Let $V\subset \mathbb C^4  \times \mathbb C ^4$ be the Zariski closure of $(t_1, J_1(t_1),J'_1(t_1),J''_1(t_1), t_2, J_2(t_2), J'_2(t_2), J''_2(t_2))$ for $(t_1,t_2) \in \mathcal C \cap (U_1 \times U_2)$. By strong minimality and transcendence of $J_1(t_1)$ and $J_2(t_2)$, $\mathbb C(V)$ is an algebraic extension of $K_1$ and of $K_2$.

Let $X_1$, $H_1$, $Y_1$, $D_1$, $X_2$, $H_2$, $Y_2$, $D_2$ be the lifts of these derivations on $\mathbb C(V)$. On this field $D_1$ and $D_2$ are colinear, indeed $D_1 = \sigma D_2$ with $\sigma = \frac{\partial t_2}{\partial t_1}$ as a function on the algebraic curve $\mathcal C$.

As  $X_1$, $H_1$, $Y_1$ and $X_2$, $H_2$, $Y_2$ are two basis of $\mathbb C(J(\mathcal C))$-derivations of $\mathbb C(V)$, one has
$$
\begin{bmatrix}
X_1\\ H_1 \\ Y_1  
\end{bmatrix}
= 
A 
\begin{bmatrix}
X_2 \\ H_2  \\ Y_2
\end{bmatrix},
$$ 
where $A$ is a $3\times 3$ matrix with coefficients in $\mathbb C(V)$. The Lie bracket of $D_1 = \sigma D_2$ with each components of 
$\begin{bmatrix}
X_1\\ H_1 \\ Y_1  
\end{bmatrix}
$ gives 
$$
\begin{bmatrix}
0\\ -D_1 \\ -t_1D_1  
\end{bmatrix}= D_1 A 
\begin{bmatrix}
X_2 \\ H_2 \\ Y_2
\end{bmatrix}
+ A \left[\sigma D_2 , \begin{bmatrix} 
 X_2 \\   H_2 \\  Y_2
\end{bmatrix}\right]
=D_1 A 
\begin{bmatrix}
X_2 \\ H_2 \\ Y_2
\end{bmatrix}
- A \begin{bmatrix}
 X_2\sigma  \\  (H_2\sigma - 1)\\  (Y_2\sigma -t_2)
 \end{bmatrix} D_2.
$$ 

By independence of the derivations, $A$ is a matrix of first integrals of $D_1$. As the vector field $D_1$ has a Zariski dense trajectory on $V$, namely the analytic curve parameterized  by $(t_1,t_2) \mapsto (t_1, J_1(t_1),J'_1(t_1),J''_1(t_1), t_2, J_2(t_2), J'_2(t_2), J''_2(t_2))$  for $(t_1, t_2) \in \mathcal C \cap U_1 \times U_2$, $A$ is a matrix of complex numbers.

On $V$, one gets $\mathbb C X_1 + \mathbb C H_1 + \mathbb C Y_1 = \mathbb C X_2 + \mathbb C H_2 + \mathbb C Y_2$.

The derivations $X_1$, $H_1$, $Y_1$, $X_2$, $H_2$ and $Y_2$ preserve $\mathbb C(\mathcal C)$ the field of rational functions on the curve $\mathcal C$ and induce vector fields on $\mathcal C$. From this, we get that $\mathcal C$ is the graph of a correspondence on $\mathbb P^1$ sending $\frac{d}{dt_1}$, $t_1\frac{d}{dt_1}$ and $\frac{t_1^2}{2}\frac{d}{dt_1}$ into $\mathbb C\frac{d}{dt_2} +  \mathbb C t_2\frac{d}{dt_2}+ \mathbb C \frac{t_2^2}{2}\frac{d}{dt_2}$. Hence it is the graph of a homography.
\end{proof}

A rational function $R_1$ on $Y_1$ is said to have a {\em finite local Galois group} at a pole $p$ if $p$ is a regular singular point of the associated second order linear differential equation $\frac{d^2 \psi}{dy^2} +\frac{1}{2}R(y) \psi =0$ with finite local monodromy at $p$. By Fuch's theory this condition is equivalent to the fact that any Schwarzian primitive\footnote{A function whose Schwarzian is $R_1$.} of $R_1$ near $p$, given by the quotient of two independant solutions $\frac{\psi_1}{\psi_2}$, belongs to $\mathcal O^{alg}_{p}$, the algebraic closure of the ring of germs of holomorphic functions at $p$.

 This condition implies that $R_1$ has a pole at $p$ of order less than or equal to $2$, i.e., $ R_1(y) = \frac{1}{2}{\frac{1-\alpha^{2}}{(y-p)^2}} + \frac{\beta}{y-p} + f(y)$ with $\alpha \in \mathbb Q ^\ast$, $\beta \in \mathbb C$ and $f$ holomorphic in a neighborhood of $p$. Note that finite local Galois group at $p$ is not equivalent to $p$ being a regular singular point with rational parameter $\alpha$, since when $\alpha \in \mathbb Z$, the solution may have a logarithmic singularity at $p$ with infinite monodromy.

Let $Y_1^\ast$ be the curve $Y_1$ punctured at poles of $R_1$. Assume $Y_1^\ast$ is hyperbolic and choose a uniformizing function $\rho_1 : \mathbb H \to Y_1^\ast$. Let $\Gamma_1 = \pi_1(Y_1^\ast)$ be the fundamental group of the complex curve $Y_1^\ast$. As $\rho_1$ is a universal covering of $Y_1^\ast$, one can see $\pi_1(Y_1^\ast)$ a Fuchsian subgroup of ${\rm PSL}_2(\mathbb R)$. We define $Y_2^\ast$ and $\Gamma_2$, given a uniformization $\rho_2$, similarly.

\begin{prop}
Assume $R_1$ and $R_2$ have no poles with finite local Galois groups. If $\mathcal{C}$ is a bi-algebraic curve with respect to $J$, then there exists a uniformisation $\rho_2 : \mathbb H \to Y^\ast_2$ such that $\Gamma_1 \sim \Gamma_2$ and $J(\mathcal C)$ is a ${\rm Comm}(\Gamma_1)$-correspondence , i.e., a correspondence in $Y_1^\ast \times Y_2^{\ast}$ which is the image under $(\rho_1, \rho_2)$ of the graph of some $g \in {\rm Comm}(\Gamma_1) (= {\rm Comm}(\Gamma_2))$.
\end{prop}

\begin{proof}
The proof consists of two lemmas.
\begin{lem}
Assume $R_1$ and $R_2$ have no poles with finite local Galois groups. If $p\in J(\mathcal C)$ is such that one of the projection on $Y_1$ or $Y_2$ ramifies, then the projection of $p$ on $Y_1$ (resp. $Y_2$) is a pole of $R_1$ (resp. $R_2$).
\end{lem}
\begin{proof}
By Theorem \ref{bialgebraic1} and using the hypothesis, we may assume that the curve $\mathcal C$ is the graph of $g\in \rm{PSL}_2(\mathbb C)$. Let  $p \in J(\mathcal C)$ be a point such that the first projection ramifies  at $p$ and its second projection is not a pole of $R_2$. Let $\tau_2$ be a Schwarzian primitive of $R_2$ near the projection of $p$. Its pull back on $J(\mathcal C)$ is an holomorphic function and its direct image near the projection of $p$ in $Y_1$ belongs to $\mathcal O^{alg}_{{\rm pr_1}(p)}$. As it is a Schwarzian primitive of $R_1$, and $R_1$ has no poles of finite local Galois group, the first projection of $p$ is not a pole of $R_1$.  The lemma is proved.
\end{proof}

\begin{lem}{\cite[page 337]{Marg}}
If a correspondence $X\subset Y_1^\ast \times Y_2^\ast$ is a covering of both factors then  then there exists a uniformization $\rho_2 : \mathbb H \to Y^\ast_2$ such that  $\Gamma_1 \sim \Gamma_2$ and $X$ is a ${\rm Comm}(\Gamma_1)$-correspondence .
\end{lem}

\begin{proof}

Let $\rho : \mathbb H \to X$ be the uniformization map of $X$ such that $\rho_1 = {\rm pr}_1 \circ \rho$. 
It defines a embedding of $\pi_1(X)$ in ${\rm PSL}_2(\mathbb R)$ with image which we denote by $\Gamma$. 
The uniformization of $Y^\ast_2$ is $\rho_2 = {\rm pr}_2 \circ \rho$ and defines a embedding of $\pi_1(Y^\ast_2)$ in ${\rm PSL}_2(\mathbb R)$ with image which we denote by $\Gamma_2$.  
As $\Gamma$ is a finite index subgroup of $\Gamma_1$ and of $\Gamma_2$, the two groups are commensurable. 

Let $Z \subset \mathbb H \times \mathbb H$ be an irreducible component of the analytic variety $(\rho_1,\rho_2)^{-1}(X)$. The subset $Z$ is a non-ramified covering of $\mathbb H$ and hence it is the graph of automorphism from $\mathbb H$ to $\mathbb H$. By the Schwarz lemma, it is the graph of an homography $g\in {\rm PSL}_2(\mathbb C)$. General arguments ensure that $g \in {\rm Comm}(\Gamma_1)$,  see for example lemmas 5.15 and 5.16 in \cite{CasFreNag}.
\end{proof}

The polar locus of $R_1$ is mapped by $J(\mathcal C)$ on the polar locus of $R_2$ and these sets contained the projection of ramification points. Hence the restriction of $J(\mathcal C)$ above $Y_1^\ast \times Y_2^\ast$ is a covering of both factors.

Applying the second lemma, one gets that $J(\mathcal C)$ is a ${\rm Comm}(\Gamma_1)$-correspondence. 
\end{proof}

We conclude this section with an example showing that for any type of singularities  (regular, irregular, finite or infinite monodromy) there exists example of couple of Schwarzian equations with bi-algebraic curves.
\begin{exam}
Let $\Phi : \mathbb C {\rm P}^1 \to \mathbb C {\rm P}^1$ be a non constant rational map with critical points $B \subset \mathbb C {\rm P}^1$ and $t$ be an affine coordinate on $\mathbb C {\rm P}^1$. For $ R$ a rational function on $\mathbb C {\rm P}^1$, let $S_1(R)$ be the set of poles with finite local Galois group and $S_2(R)$ be the set of other poles.
Consider the Schwarzian equations with $R_1 = R$ and $R_2 = R_\Phi = R \circ \Phi (\frac{d\Phi}{dt})^2 + S_t\Phi$.

One has $S_2(R_\Phi) = \Phi^{-1}(S_2(R))$ and $S_1(R_\Phi) = (\Phi^{-1}(S_1(R)) \cup B) \setminus S_2(R_\Phi)$. If $J_2$ is a solution of the second Schwarzian equation, then $J_1 = \Phi(J_2)$ is a solution of first the Schwarzian equation. So, the diagonal in $\mathbb H \times \mathbb H$ is sent by $(J_1,J_2)$ to the graph of $\Phi$. 

\end{exam}



\end{document}